\documentclass[final,leqno,onefignum,onetabnum]{siamltex}
\usepackage[letterpaper,top=1.5in, bottom=1.5in, left=1in, right=1in]{geometry}

\usepackage{epsfig,epsf,fancybox}
\usepackage{amsmath}
\usepackage{mathrsfs}
\usepackage{amssymb}
\usepackage{amsfonts}
\usepackage{graphicx}
\usepackage{color}
\usepackage[linktocpage,colorlinks,linkcolor=blue,anchorcolor=blue,citecolor=magenta,urlcolor=cyan,hypertexnames=false]{hyperref}
\usepackage{boxedminipage}
\usepackage{stmaryrd}
\usepackage{multirow}
\usepackage{subfig}
\usepackage{booktabs}
\usepackage{algorithm}
\usepackage{algpseudocode}% http://ctan.org/pkg/algorithmicx
\usepackage{cite}
\usepackage{float}
\usepackage{tikz}
\usepackage{xcolor}
\usepackage{tcolorbox}
\usepackage{braket,mathdots}
\usepackage{mathtools}
\usepackage[thinlines]{easytable}
\usepackage{array}

\allowdisplaybreaks[4]
\usepackage{mathtools}

\newcommand{\vxi}{{\boldsymbol{\xi}}}

\newcommand{\email}[1]{\href{mailto:#1}{#1}}

\newcommand*{\QED}{\null\nobreak\hfill\ensuremath{\square}}

\usepackage{mathtools}
\usepackage[normalem]{ulem} % to use \sout

\newcommand{\bm}[1]{\boldsymbol{#1}}

\newcommand{\va}{{\mathbf{a}}}
\newcommand{\vb}{{\mathbf{b}}}
\newcommand{\vc}{{\mathbf{c}}}
\newcommand{\vd}{{\mathbf{d}}}
\newcommand{\ve}{{\mathbf{e}}}
\newcommand{\vf}{{\mathbf{f}}}

\newcommand{\vp}{{\mathbf{p}}}
\newcommand{\vq}{{\mathbf{q}}}
\newcommand{\vr}{{\mathbf{r}}}
\newcommand{\vs}{{\mathbf{s}}}

\newcommand{\vu}{{\mathbf{u}}}
\newcommand{\vv}{{\mathbf{v}}}
\newcommand{\vw}{{\mathbf{w}}}
\newcommand{\vx}{{\mathbf{x}}}
\newcommand{\vy}{{\mathbf{y}}}
\newcommand{\vz}{{\mathbf{z}}}

\newcommand{\vA}{{\mathbf{A}}}

\newcommand{\vD}{{\mathbf{D}}}

\newcommand{\vF}{{\mathbf{F}}}
\newcommand{\vG}{{\mathbf{G}}}

\newcommand{\vI}{{\mathbf{I}}}
\newcommand{\vJ}{{\mathbf{J}}}

\newcommand{\vQ}{{\mathbf{Q}}}

\newcommand{\vV}{{\mathbf{V}}}
\newcommand{\vW}{{\mathbf{W}}}

\newcommand{\cD}{{\mathcal{D}}}

\newcommand{\cO}{{\mathcal{O}}}

\newcommand{\vareps}{\varepsilon}

\newcommand{\RR}{\mathbb{R}} % real
\newcommand{\vzero}{\mathbf{0}} % 0 vector
\newcommand{\dist}{\mathrm{dist}}    % distance
\newcommand{\prox}{{\mathbf{prox}}} % proximal map
\newcommand{\dom}{{\mathrm{dom}}} % domain
\newcommand{\ip}[2]{\langle #1 , #2 \rangle}

\DeclareMathOperator*{\argmin}{arg\,min} % argmin
\DeclareMathOperator*{\argmax}{arg\,max} % argmax
\DeclareMathOperator*{\Argmin}{Arg\,min} % Argmin
\newcommand{\bc}{\begin{center}}
\newcommand{\ec}{\end{center}}

\newcommand{\bdm}{\begin{displaymath}}
\newcommand{\edm}{\end{displaymath}}

\newcommand{\beq}{\begin{equation}}
\newcommand{\eeq}{\end{equation}}

\newcommand{\bfl}{\begin{flushleft}}
\newcommand{\efl}{\end{flushleft}}

\newcommand{\bt}{\begin{tabbing}}
\newcommand{\et}{\end{tabbing}}

\newcommand{\beqn}{\begin{eqnarray}}
\newcommand{\eeqn}{\end{eqnarray}}

\newcommand{\beqs}{\begin{align*}} % no equation numbers
\newcommand{\eeqs}{\end{align*}}  % no equation numbers

\newtheorem{assumption}{Assumption}

\usepackage{enumitem}
\numberwithin{equation}{section}
\numberwithin{theorem}{section}

\newcommand{\YX}[1]{\textcolor{red}{(YX: #1)}}
\newcommand{\xu}[1]{\textcolor{black}{#1}}

\author{
    Yangyang Xu\thanks{Department of Mathematical Sciences,
    Rensselaer Polytechnic Institute, Troy, NY 12180 USA 
    (\email{xuy21@rpi.edu}, \email{yingx@rpi.edu}).}
    \and    
    Xinhao Ying\footnotemark[1]
}
\begin{document}
\title{First-Order Methods for Solving Convex (Strongly) Concave Minimax Problems with Functional Constraints}
	\date{\today}
		\maketitle
\begin{abstract}
%Convex-concave 
Minimax problems arise in many applications, including robust learning and Stackelberg models. Most  
existing methods for minimax problems address unconstrained or projection-friendly settings, while functional constrained minimax problems remain far less explored. We study a class of  
convex-(strongly-)concave minimax problems with functional constraints. By exploiting strong duality, we incorporate the inner-maximization functional constraints into the objective. This allows us to efficiently obtain inexact gradients of the primal function of the reformulation and to design a proximal augmented Lagrangian method (PALM). Each PALM subproblem is solved by an inexact accelerated proximal gradient scheme to handle inexact gradients arising from approximately solving an auxiliary maximization subproblem. We show that the proposed method returns an $\varepsilon$-KKT point and a primal $\varepsilon$-optimal solution, with $\tilde{\mathcal{O}}(\varepsilon^{-1})$ first-order oracle and iteration complexity in the convex-strongly-concave case. For the convex-concave case, the complexity remains the same for the primal gradient evaluation but increases to $\tilde{\mathcal{O}}(\varepsilon^{-\frac{3}{2}})$ for the dual part.
\end{abstract}        
\section{Introduction}
Minimax problems have been %extensively 
studied for decades across mathematics, economics, and computer science, with %classical 
close connections to game theory~\cite{rubinstein2007theory,bacsar1998dynamic,nisan2007algorithmic}. In recent years, they arise in various modern applications such as robust optimization~\cite{ben1998robust,bertsimas2004price} and adversarial learning~\cite{goodfellow2014generative,madry2017towards}. First-order methods have become workhorse for large-scale instances of such problems. However, most existing works on minimax optimization deal with unconstrained formulations or problems with easy-to-project constraints.

In this paper, we consider convex-concave minimax problems with functional constraints, formulated as
\begin{equation}\label{Basic Formulation}
    \min_{\mathbf{x}\in\RR^n}\max_{\mathbf{y}\in\RR^m}F(\mathbf{x}, \mathbf{y}) - h(\mathbf{y})+g(\mathbf{x}), \quad \text{s.t.} \quad \mathbf{c}(\mathbf{x}) \leq \vzero, \quad \mathbf{p}(\mathbf{y}) \leq \vzero.
\end{equation}
Here, $F(\vx,\vy):\RR^n\times\RR^m \to \RR$ is convex in $\vx$, concave in $\vy$ and has Lipschitz continuous gradient, $g:\RR^n\to\RR$ and $h:\RR^m\to\RR$ are Lipschitz continuous, proper, closed, convex functions with compact domains, and each component of $\vc:\RR^n\to\RR^q$ and $\vp:\RR^m\to\RR^r$ is convex and has Lipschitz continuous gradient.
%While a large body of first-order algorithms exists for unconstrained or simply constrained saddle problems, efficiently handling functional constraints with explicit gradient-complexity guarantees remains challenging.

%To deal with the above issues explicitly, in this paper, we focus on the following expectation constrained minimax stochastic optimization problem,
\subsection{Motivating Applications}
Problems in the form of~\eqref{Basic Formulation} arise in many areas. A couple of motivating applications are presented below.
\subsubsection{AUC Maximization with Fairness Constraints}

AUC (Area Under the ROC Curve) is a widely used performance measure for binary classification, defined as the probability that a classifier ranks a randomly chosen positive instance higher than a randomly chosen negative one.

Let the feature space be $\mathcal{W}\subseteq \RR^n$, the response space $\mathcal{Y}=\{+1,-1\}$, and the training dataset $\mathcal{D}$ %$\mathcal{D}\coloneqq\{(\vv_i,y_i)\}_{i=1}^n$ 
consist of i.i.d. samples drawn from an unknown distribution on $\mathcal{W}\times\mathcal{Y}$. The AUC maximization problem~\cite{yang2025data} can be expressed as %in the following expectation from
\begin{equation*}
    \max_{\vx\in \RR^d}\mathbb{E}_{(\vv,y)\in \mathcal{D},(\vv',y')\in \mathcal{D}}\left[\mathbb{I}_{[\vx^\top \vv>\vx^\top \vv']}|y=1,y'=-1\right] = \min_{\vx,a,b} \max_{\alpha} 
    \left\{
        \mathbb{E}_{(\vv,y)\in \mathcal{D}}
        \left[
            f(\vx,a,b,\alpha;(\vv,y))
        \right]
    \right\},
\end{equation*}
where $\mathbb{I}$ is the 0-1 indicator function, %. Yang et al.~\cite{yang2025data} reformulated the objective into a stochastic minimax problem:
% \begin{equation*}
%     \min_{\vx,a,b} \max_{\alpha} 
%     \left\{
%         \mathbb{E}_{(\vv,y)\in \mathcal{D}}
%         \left[
%             f(\vx,a,b,\alpha;(\vv,y))
%         \right]
%     \right\},
% \end{equation*}
% where 
and
\begin{equation*}
    \begin{aligned}
f(\vx,a,b,\alpha;(\vv,y))
&= (1-p)(\vx^\top \vv - a)^2 \mathbb{I}_{[y=1]}
   + p(\vx^\top \vv - b)^2 \mathbb{I}_{[y=-1]} \\
&\quad + 2(1+\alpha)
   \left(
     p\,\vv^\top \vx \mathbb{I}_{[y=-1]}
     - (1-p)\,\vv^\top \vx \mathbb{I}_{[y=1]}
   \right)
   - p(1-p)\alpha^2.
\end{aligned}
\end{equation*}
%Based on this stochastic minimax formulation, fairness constraints can be incorporated. Yang et al.~\cite{yang2025data} adopted the equivalence condition in~\cite{zafar2019fairness} and formulated 
To encourage fairness in the trained model, fairness constraints are introduced in \cite{zafar2019fairness}, and an equivalent condition is adopted in \cite{yang2025data} as %follows:
%\begin{equation*}
    $\mathbb{E}_{(\vv,u)\in\mathcal{D}_s}[(u-\bar{u})\vv^\top\vx]\leq c,\, \mathbb{E}_{(\vv,u)\in\mathcal{D}_s}[(u-\bar{u})\vv^\top\vx]\geq -c,$ 
%\end{equation*}
where $u\in\{0,1\}$ is the sensitive-group label, $\bar{u}$ represents the average of $u$, and $c\geq0$ is a tolerance parameter. This constraint limits the correlation between the prediction $\vv^\top\vx$ and the sensitive attribute $u$.

%Since the dataset is discrete and finite,
For discrete data sets $\cD$ and $\cD_s$, we replace the population expectations by their empirical approximations and obtain
\begin{align*}
    \min_{\vx,a,b} \max_{\alpha}&\frac{1}{|\cD|} 
        \sum_{(\vv,y)\in\cD}
            f(\vx,a,b,\alpha;(\vv,y))
,\\
    \text{s.t.}\; & \frac{1}{|\cD_s|}\sum_{(\vv,u)\in\cD_s} [(u-\bar{u})\vv^\top\vx]\leq c,\ 
     \frac{1}{|\cD_s|}\sum_{(\vv,u)\in\cD_s} [(u-\bar{u})\vv^\top\vx]\geq -c,
\end{align*}
% \begin{align*}
%     \min_{\vx,a,b} \max_{\alpha}&\frac{1}{n} 
%         \sum_{i=1}^n
%             f(\vx,a,b,\alpha;(\vv_i,y_i))
% ,\\
%     \text{s.t.}\; & \frac{1}{n}\sum_{i=1}^n [(u_i-\bar{u})\vv_i^\top\vx]\leq c,\ 
%      \frac{1}{n}\sum_{i=1}^n [(u_i-\bar{u})\vv_i^\top\vx]\geq -c.
% \end{align*}
%In the above formulation, $(\vx,a,b)$ form the minimizing variable, $\alpha$ is the maximizing variable, and the fairness constraints appear as constraints. Hence, this problem 
which is a special case of~\eqref{Basic Formulation} without a functional constraint on the maximizing variable. %$\vp(\cdot)\equiv \vzero$.

\subsubsection{Robust Portfolio Selection}
Portfolio selection is about how to allocate capital among available assets to balance return and risk. Since the market parameters used to evaluate return and risk are estimated from data, small estimation errors may significantly affect the selected portfolio. To address this issue, \cite{goldfarb2003robust} studies robust portfolio models in which market-parameter perturbations are %modeled as 
unknown but bounded, and portfolios are optimized against the worst-case behavior of these perturbations. Let $\vr\in\RR^n$ denote the vector of asset returns and satisfy %, which is assumed to be a
%random variable given by
$
\vr=\bm{\mu}+\vV^\top \vf+\bm{\varepsilon}.
$ 
Here, $\bm{\mu}\in\RR^n$ is the vector of mean returns, $\vf\sim\mathcal{N}(\vzero,\vF)\in \RR^m$ is the vector of common factor returns that affect asset returns, $\vV\in \RR^{m\times n}$ is the matrix of factor loadings, describing the sensitivities of the assets to the common factors, and $\bm{\varepsilon}\sim\mathcal{N}(\vzero,\vD)$ is the vector of residual returns not explained by the common factors. %In addition, 
Assume that $\bm{\varepsilon}$ is independent of $\vf$, $\vF\succ \vzero$ and $\vD=\operatorname{diag}(\vd)\succeq\vzero$. Let $\bm{\lambda}\in\RR^n$ %denote the portfolio vector, where 
with ${\lambda}_i$ representing the fraction of total wealth invested in asset $i$. Then the return of portfolio $\bm{\lambda}$ is given by %$r_{\lambda}$.
$
r_{\lambda}=\vr^\top\bm{\lambda}\sim \mathcal{N}\left(\bm{\mu}^\top\bm{\lambda},\bm{\lambda}^\top(\vV^\top\vF\vV+\vD )\bm{\lambda}\right).
$

The robust maximum return problem in \cite{goldfarb2003robust} maximizes the worst-case expected return while requiring the worst-case variance to stay below a prescribed risk level:
\begin{equation}\label{eq:portfolio-basic}
    \max_{\bm{\lambda}}
\min_{\bm{\mu}\in S_m}
\bm{\mu}^\top\bm{\lambda},
\quad
\text{s.t.}\quad    
\max_{\vV\in S_v,\ \vD\in S_d}
\bm{\lambda}^\top(\vV^\top\vF\vV+\vD )\bm{\lambda}\leq \nu,\ \mathbf 1^\top\bm{\lambda}=1,\ \bm{\lambda}\geq \vzero.
\end{equation}
% $$
% \min_{\vV\in S_v,\ \vD\in S_d}
% \max_{\bm{\lambda}}
% \min_{\bm{\mu}\in S_m}
% \bm{\mu}^\top\bm{\lambda},
% \quad
% \text{s.t.}\quad
% \bm{\lambda}^\top(\vV^\top\vF\vV+\vD )\bm{\lambda}\leq \nu,\ \mathbf 1^\top\bm{\lambda}=1,\ \bm{\lambda}\geq \vzero.
% $$
% The robust minimum variance portfolio selection problem in \cite{goldfarb2003robust} minimizes the worst-case variance of the portfolio subject to the constraint that the worst-case expected return of the portfolio is at least a prescribed return level.
% $$
% \min_{\bm{\lambda}}
% \max_{\vV\in S_v,\ \vD\in S_d}\bm{\lambda}^\top(\vV^\top\vF\vV+\vD )\bm{\lambda},
% \quad
% \text{s.t.}\quad 
% \min_{\bm{\mu}\in S_m}\bm{\mu}^\top\bm{\lambda}\geq \nu,\ \mathbf 1^\top\bm{\lambda}=1,\ \bm{\lambda}\geq \vzero.
% $$
% $$
% \max_{\bm{\mu}\in S_m}
% \min_{\bm{\lambda}}
% \max_{\vV\in S_v,\ \vD\in S_d}\bm{\lambda}^\top(\vV^\top\vF\vV+\vD )\bm{\lambda},
% \quad
% \text{s.t.}\quad 
% \bm{\mu}^\top\bm{\lambda}\geq \nu,\ \mathbf 1^\top\bm{\lambda}=1,\ \bm{\lambda}\geq \vzero.
% $$
Here, $\bm{\mu}$, $\vV$, and $\vD$ lie in the uncertainty sets $S_m$, $S_v$, and $S_d$, respectively, given by %below 
%The sets $S_m$ and $S_v$ are induced by the linear-regression estimates of the asset mean returns and factor loadings, while $S_d$ bounds the residual variances.
\begin{align*}
      &S_v=\left\{\vV:\ \vV=\vV_0+\vW,\ \|\vW_i\|_g\le \rho_i,\ i=1,\ldots,n\right\},\\
      &S_d=\left\{\vD:\ \vD=\operatorname{diag}(\vd),\ d_i\in[\underline d_i,\bar d_i],\ i=1,\ldots,n\right\},\ 
      % S_m=\left\{\bm{\mu}:\ \bm{\mu}=\bm{\mu}_0+\bm{\xi},\ |\xi_i|\le \gamma_i,\ i=1,\ldots,n\right\}.\\
      S_m=\left\{\bm{\mu}:\ \bm{\mu}=\bm{\mu}_0+\bm{\xi},\ \|\bm{\xi}\|_a\le {\gamma} \right\},
\end{align*}
where \xu{$\vW_i$ is the $i$-th column of $\vW$,} $\|\vw\|_g:=\sqrt{\vw^\top\vG\vw}$ and $\|\bm{\xi}\|_a:=\sqrt{\bm{\xi}^\top\vA\bm{\xi}}$ denote the weighted %elliptic 
norms with respect to 
symmetric positive definite matrices $\vG$ and $\vA$, respectively.

Define $\bar{\vD}=\operatorname{diag}(\bar{d_1},\ldots,\bar{d_n})$ and
\begin{align*}
    \phi(\bm{\lambda})=\max_{\|\vw\|_g\leq \bm{\rho}^\top\bm{\lambda}}(\vV_0\bm{\lambda}+\vw)^T\vF(\vV_0\bm{\lambda}+\vw)+\bm{\lambda}^\top\bar{\vD}\bm{\lambda},\ 
    \Lambda\coloneqq\left\{\bm{\lambda}\in \RR^n:\mathbf 1^\top\bm{\lambda}=1,\ \bm{\lambda}\geq \vzero\right\}.
\end{align*}
Let $\mathbb{I}$ be the 0-$\infty$ indicator function. Then the problem~\eqref{eq:portfolio-basic} can be written as
\begin{align}\label{eq:robust-pt}
\min_{\bm{\lambda}}
\max_{\bm{\mu}}
-\bm{\mu}^\top\bm{\lambda}+\mathbb{I}_{\Lambda}(\bm{\lambda})-\mathbb{I}_{S_m}(\bm{\mu}),\quad
\text{s.t.}\quad
\phi(\bm{\lambda})\leq \nu,
\end{align}
%Taking $\bm{\lambda}$ as the minimizing variable and $\bm{\mu}$ as the maximizing variable, this formulation 
which fits the structure in problem~\eqref{Basic Formulation}.

\xu{Clearly the objective in \eqref{eq:robust-pt} is convex concave about $\bm{\lambda}$ and $\bm{\mu}$. In addition, }%The %value function 
$\phi(\bm{\lambda})$ is convex in $\bm{\lambda}$, but it may not %necessarily 
be smooth. %in general. 
In the special case of $\vF=\alpha \vG$ for some $\alpha>0$, as considered in \cite[Eqns.~(64)--(68)]{goldfarb2003robust}, it holds that $\phi(\bm{\lambda})=\alpha(\|\vV_0\bm{\lambda}\|_G+\bm{\rho}^\top\bm{\lambda})^2+\bm{\lambda}^\top\bar{\vD}\bm{\lambda}.$ Under a nondegeneracy condition such as $\vV_0\bm{\lambda}$ uniformly away from $\vzero$ on the feasible set $\Lambda$, then %gives a smooth representation of 
$\phi$ is smooth on $\Lambda$. %Alternatively, 
Otherwise, one may consider the smoothed relaxation $\phi_\delta(\bm{\lambda})=\alpha(\sqrt{\|\vV_0\bm{\lambda}\|_G^2+\delta^2}+\bm{\rho}^\top\bm{\lambda})^2+\bm{\lambda}^\top\bar{\vD}\bm{\lambda}$ %which satisfies the smoothness requirement 
for a certain small number $\delta>0$. 
%where $\bar{\vD}=\operatorname{diag}(\bar{\vd_1},\ldots,\bar{\vd_n})$.

\xu{In general, $\phi$ cannot be evaluated in a closed form. Nevertheless, it can be computed to a high accuracy.} 
{\color{black}Let $\vs=\frac{\vG^{\frac{1}{2}}\vw}{\bm{\rho}^\top\bm{\lambda}}$, then %$\phi(\bm{\lambda})$ can be rewritten as 
$$\phi(\bm{\lambda})=\max_{\|\vs\|_2\leq 1}(\vG^{\frac{1}{2}}\vV_0\bm{\lambda}+(\bm{\rho}^\top\bm{\lambda})\vs)^T\vG^{-\frac{1}{2}}\vF\vG^{-\frac{1}{2}}(\vG^{\frac{1}{2}}\vV_0\bm{\lambda}+(\bm{\rho}^\top\bm{\lambda})\vs)+\bm{\lambda}^\top\bar{\vD}\bm{\lambda}.$$
\xu{Suppose $\vs^*$ is a maximizer and $\eta$ its corresponding Lagrangian multiplier. Then through formulating the KKT system, it is not difficult to derive}
\begin{align*}
\vs^*=\vQ\left(\eta\vI-(\bm{\rho}^\top\bm{\lambda})^2\Sigma\right)^{-1}(\bm{\rho}^\top\bm{\lambda})\Sigma\vQ^\top\vG^{\frac{1}{2}}\vV_0\bm{\lambda},\quad \sum_{i=1}^m\left(\frac{(\bm{\rho}^\top\bm{\lambda})\sigma_i}{\eta-(\bm{\rho}^\top\bm{\lambda})^2\sigma_i}\right)^2\left(\vQ^\top\vG^{\frac{1}{2}}\vV_0\bm{\lambda}\right)_i^2=1,    
\end{align*}
\xu{where we perform one-time eigen-decomposition $\vG^{-\frac{1}{2}}\vF\vG^{-\frac{1}{2}}=\vQ\Sigma\vQ^\top$ with $\Sigma=\operatorname{diag}(\sigma_1,\ldots,\sigma_m)$. Then Newton's method can be applied to find a high-accurate solution for $\eta$ and thus $\vs^*$.}
% The Lagrangian function for the inner maximization is $$\mathcal{L}(\vs,\eta)=\left(\vG^{\frac{1}{2}}\vV_0\bm{\lambda}+(\bm{\rho}^\top\bm{\lambda})\vs\right)^T\vG^{-\frac{1}{2}}\vF\vG^{-\frac{1}{2}}\left(\vG^{\frac{1}{2}}\vV_0\bm{\lambda}+(\bm{\rho}^\top\bm{\lambda})\vs\right)-\eta(\|\vs\|_2^2-1).$$
% The KKT stationary condition gives
% \begin{equation*}
%     (\bm{\rho}^\top\bm{\lambda})\vG^{-\frac{1}{2}}\vF\vG^{-\frac{1}{2}}\left(\vG^{\frac{1}{2}}\vV_0\bm{\lambda}+(\bm{\rho}^\top\bm{\lambda})\vs\right)-\eta \vs=0.
% \end{equation*}
% Let $$\vG^{-\frac{1}{2}}\vF\vG^{-\frac{1}{2}}=\vQ\Sigma\vQ^\top,\ \Sigma=\operatorname{diag}(\sigma_1,\ldots,\sigma_m).$$ Then the maximizer satisfies
% \begin{equation*}
%     \vs^*=\vQ\left(\eta\vI-(\bm{\rho}^\top\bm{\lambda})^2\Sigma\right)^{-1}(\bm{\rho}^\top\bm{\lambda})\Sigma\vQ^\top\vG^{\frac{1}{2}}\vV_0\bm{\lambda}.
% \end{equation*}
% Combining this expression with the boundary condition $\|\vs^*\|_2=1$, the multiplier $\eta$ is determined by the scalar equation
% \begin{equation*}
% \begin{aligned}
%     (\vQ^\top\vs^*)_i=\frac{(\bm{\rho}^\top\bm{\lambda})\sigma_i}{\eta-(\bm{\rho}^\top\bm{\lambda})^2\sigma_i}\left(\vQ^\top\vG^{\frac{1}{2}}\vV_0\bm{\lambda}\right)_i,\ i=1,\ldots,m,\\
%     \sum_{i=1}^m\left(\frac{(\bm{\rho}^\top\bm{\lambda})\sigma_i}{\eta-(\bm{\rho}^\top\bm{\lambda})^2\sigma_i}\right)^2\left(\vQ^\top\vG^{\frac{1}{2}}\vV_0\bm{\lambda}\right)_i^2=1.
% \end{aligned}
% \end{equation*}
}

% Let $\mathbb{I}$ be the indicator function. Then the problem~\eqref{eq:portfolio-basic} can be written as
% \begin{align*}
% \min_{\bm{\lambda}}
% \max_{\bm{\mu}}
% -\bm{\mu}^\top\bm{\lambda}+\mathbb{I}_{\Lambda}(\bm{\lambda})-\mathbb{I}_{S_m}(\bm{\mu}),\quad
% \text{s.t.}\quad
% \phi(\bm{\lambda})\leq \nu.
% \end{align*}
% Taking $\bm{\lambda}$ as the minimizing variable and $\bm{\mu}$ as the maximizing variable, this formulation fits the structure in problem~\eqref{Basic Formulation} by setting
% \begin{align*}
%     \vx=\bm{\lambda},\ \vy=\bm{\mu},\ F(\vx,\vy)=-\vy^\top \vx,\\
%     g(\vx)=\mathbb{I}_{\Lambda}(\vx),\ h(\vy)=\mathbb{I}_{S_m}(\vy),\ \vc(\vx)=\phi(\vx)-\nu,\ \vp(\vy)=\vzero.
% \end{align*}
\subsection{Related Work}
Though our focus is on the minimax problem \eqref{Basic Formulation}, our algorithm will be designed based on a functional constrained minimization reformulation. Hence, we review deterministic first-order methods for solving minimax problems with or without functional constraints as well as for minimization problems with functional constraints. %, including classical methods, complexity results under different structural assumptions, and existing works on minimax problems with functional constraints. We then briefly discuss related results on convex minimization with functional constraints.
\subsubsection{Minimax Problem}
The minimax problem without functional constraints has been extensively studied, such as in \cite{lin2020near,mokhtari2020convergence,jiang2025generalized, mokhtari2020unified,yan2020optimal,bullins2022higher} for (strongly-)convex (strongly-)concave problems, in \cite{nemirovski2004prox,zhang2021complexity,lin2025two,lin2020gradient,xu2023unified,luo2020stochastic,chen2021proximal,guo2023fast,nouiehed2019solving} for nonconvex (strongly-)concave problems, and in \cite{zheng2023universal,diakonikolas2021efficient,yang2020global} for nonconvex-nonconcave problems with certain special structures. Classical approaches for solving minimax problems include %proximal and 
mirror-prox methods~\cite{nemirovski2004prox}, extragradient methods~\cite{korpelevich1976extragradient, mokhtari2020unified}, primal-dual methods~\cite{chambolle2011first, malitsky2018first}. More recent first-order methods also include gradient descent-ascent and optimistic gradient schemes~\cite{mokhtari2020convergence, jiang2025generalized}. %Depending on assumptions on the objective, 
%Among extensive existing works, many papers focus ~; many other papers consider ~; some have also studied nonconvex-nonconcave minimax problems~ under certain structural assumptions. 
%In particular, 
For smooth convex-concave minimax problems, the optimal complexity of deterministic first-order methods is known to scale as $\mathcal{O}(\varepsilon^{-1})$\cite{lin2020near, hamedani2021primal} to produce an $\varepsilon$-optimal solution. For nonconvex minimax problems, existing methods aim at producing an $\varepsilon$-stationary solution, and the complexity is higher. We refer the reader to the survey paper~\cite{razaviyayn2020nonconvex} and the references therein for broader overviews on nonconvex minimax problems.

%However, 
Compared with the extensive literature on minimax optimization without functional constraints, the study of minimax problems with functional constraints remains much more limited. %Among the methods discussed above, many classical approaches also accommodate separable domain constraints of the form $\vx\in\mathcal{X}$ and $\vy\in\mathcal{Y}$ where $\mathcal{X}$ and $\mathcal{Y}$ are simple closed convex sets, or more generally are equipped with efficiently computable projection, proximal, or resolvent-type steps~\cite{nemirovski2004prox,chambolle2011first}. Such constraints can therefore be handled implicitly through projection, proximal, or resolvent-type steps, rather than treated as explicit functional constraints.
Among the existing several works on minimax problems with explicit functional constraints, most of them consider coupled constraints $\vq(\vx,\vy)\leq \vzero$, which are more general than those in \eqref{Basic Formulation} that we consider. %and  class of problems and are also . 
Even when the objective is convex concave, a coupled constrained minimax problem can be intractable as shown in \cite{tsaknakis2023minimax} and thus more difficult than the problem \eqref{Basic Formulation}. In~\cite{LuMei2025}, a first-order augmented Lagrangian method (ALM) is given for nonconvex-concave minimax problems with coupled constraints. %, together with complexity guarantees for finding $\varepsilon$-KKT solutions. 
In~\cite{goktas2021convex}, max-oracle gradient descent and nested gradient descent/ascent methods are presented for convex-concave minimax problems with coupled constraints. In~\cite{dai2024optimality}, a proximal gradient multistep ascent-descent method is proposed for nonsmooth convex-concave minimax problems with coupled linear equality constraints $A\vx+B\vy+\vc=\vzero$. %, together with complexity guarantees for finding $\varepsilon$-stationary points. 

While the existing literature on explicit functional constraints mainly focuses on the coupled setting, \cite{yang2025data} provides a closely related formulation by considering stochastic convex-concave minimax problems with separable expectation constraints. It proposes a stochastic primal-dual subgradient method with the %statistically 
optimal complexity $\mathcal{O}(\varepsilon^{-2})$ \xu{to reduce all of the expected objective gap, duality gap, and constraint violation to~$\varepsilon$}. However, the algorithm in \cite{yang2025data} does not exploit the potential smoothness structure of the problem, its complexity remains $\mathcal{O}(\varepsilon^{-2})$ even if deterministic subgradients are used. In contrast, our method fully exploits the convexity-concavity structure as well as the smoothness and is able to achieve lower complexity results.
%Motivated by this model, we consider deterministic minimax problems with separable functional constraints, for which we develop a different algorithmic framework.

\subsubsection{Minimization Problem with Functional Constraints}
First-order methods for convex minimization with functional constraints have also been extensively studied such as in \cite{lu2023iteration, xu2021iteration-ialm, xu2021first-ALM, zhang2022solving, lan2020algorithms, li2021inexact, yan2022adaptive, lin2018level-SIOPT}. %In contrast to problems with unconstrained or easy-to-project constraints, 
Due to the challenge of performing direct projection, functional constraints are often handled by penalty, or augmented Lagrangian (AL) \cite{hestenes1969multiplier,powell1969method,rockafellar1976augmented}, or related primal-dual approaches~\cite{xu2020primal}. %Classical augmented Lagrangian methods trace back to the multiplier methods of Hestenes, Powell, and Rockafellar~\cite{hestenes1969multiplier,powell1969method,rockafellar1976augmented}, while 
Recent representative first-order methods include linearized ALM with $\mathcal{O}(\varepsilon^{-1})$ gradient-evaluation complexity for obtaining an $\varepsilon$-optimal solution~\cite{xu2021first-ALM} and an inexact proximal ALM with %$\mathcal{O}(\varepsilon^{-\frac{7}{4}})$ 
$\tilde{\mathcal{O}}(\varepsilon^{-1})$ first-order iteration complexity for finding an $\varepsilon$-KKT point \cite{lu2023iteration}. The order of $\varepsilon^{-1}$ is not improvable in general as it matches the lower complexity bound established in \cite{ouyang-xu2021lower}. First-order methods have also been designed and analyzed for nonconvex minimization problems with functional constraints, based on a penalty framework \cite{lin2022complexity, liu2025single, wang2017penalty}, or a (proximal) AL framework \cite{dahal2026damped, li2021rate-ialm, li2021augmented, shi2026momentum, kong2023iteration, kong2023iteration-MOR}, or sequential quadratic programming \cite{curtis2024sequential, curtis2024worst}, or a proximal-point framework \cite{boob2023stochastic, boob2025level}. Exhausting all the papers is impossible. We refer the readers to the cited papers and the references therein. %improved to $\tilde{\mathcal{O}}(\varepsilon^{-1})$ by an adaptively regularized variant~\cite{lu2023iteration}. %In~\cite{zhang2022solving}, matching upper and lower oracle complexity bounds for smooth convex optimization with functional constraints are provided. In particular, the oracle complexity in the convex case is $\mathcal{O}(\varepsilon^{-\frac{1}{2}})$, where the oracle returns the function values and gradients of both the objective and constraint functions. From the primal perspective, the proposed algorithm attains a primal $\varepsilon$-optimal point within $\tilde{\mathcal{O}}(\varepsilon^{-1})$ gradient and proximal-mapping calls.

\subsection{Contributions}
Our main contributions lie in designing a novel first-order method and establishing its complexity results for solving a class of convex-(strongly-)concave minimax problems with convex functional constraints on both primal and dual variables. To the best of our knowledge, this is the first work that explicitly studies \emph{deterministic} minimax problems with explicit decoupling %separable 
functional constraints and establishes first-order complexity guarantees to \textbf{near global optimality}. To address the challenges caused by dual functional constraints on performing inner maximization, we \emph{first}  reformulate the considered problem into a minimax problem with functional constraints only on minimization variables, by introducing Lagrangian multiplier to the inner functional constrained maximization and utilizing the strong duality condition. This reformulation is significant as it enables us to very efficiently solve the inner maximization problem without functional constraints to desired high accuracy such that an inexact gradient of the primal function is easily available. \emph{Second}, we design a first-order proximal ALM framework to solve the primal functional constrained minimization problem of the reformulation, with each proximal ALM subproblem solved by an inexact version of Nesterov's accelerated proximal gradient method. Both inner and outer loops are terminated based on computationally checkable stopping conditions. \emph{Third}, we establish first-order complexity results of the proposed method. For convex-strongly-concave cases, it achieves nearly-optimal order $\tilde\cO(\varepsilon^{-1})$ of complexity  on both primal and dual first-order oracles and iterations, to produce either an $\varepsilon$-KKT point or an $\varepsilon$-optimal solution; for convex-concave cases, the complexity remains $\tilde\cO(\varepsilon^{-1})$ for primal iterations and increases to $\tilde\cO(\varepsilon^{-\frac{3}{2}})$ for dual iterations. In both cases, our results are better than that in \cite{yang2025data}, which does not exploit the smoothness or deterministic structures.

\subsection{Notations and Definitions}\label{se:notation}
%$\ip{\cdot}{\cdot}$ denotes the Euclidean inner product. For vectors, 
$\|\cdot\|$ denotes the Euclidean norm for vectors, and the operator norm for matrices. $\vJ_\vp(\vy)$ is the Jacobian of $\vp$ at $\vy$. For a set $\mathcal{Y}\subset \RR^m$, %the distance from a point $\vx\in\RR^m$ to $\mathcal{Y}$ is defined as 
$\dist(\vx,\mathcal{Y})\coloneqq\min_{\vy\in\mathcal{Y}}\|\vx-\vy\|$ for any $\vx\in\RR^m$. Define $[\va]_+\coloneqq\max\{\va,0\}$, $\forall\va\in\RR$.

We will be interested in an $\varepsilon$-KKT point of \eqref{Basic Formulation} as well as its primal $\varepsilon$-optimal solution. 
\begin{definition}[$\varepsilon$-KKT point]\label{def1}
Given $\varepsilon\ge 0$, a point $(\vx, \vy)$ is called an {$\varepsilon$-KKT point} of problem \eqref{Basic Formulation} if there are
$\mathbf z\geq \vzero$ and $\mathbf u\ge\vzero$ such that
\begin{align*}
&  \dist\left(\vzero,\nabla_{ \vx}F(\mathbf{x},\mathbf{y})+\partial g(\mathbf{x})+\mathbf{J}^\top_\mathbf{c}(\mathbf{x})\mathbf{z}\right)\leq \varepsilon,\ \|[\mathbf{c}(\mathbf{x})]_+\|\leq \varepsilon,\ \textstyle\sum_{i=1}^q|\mathbf{c}_i(\mathbf{x})\mathbf{z}_i|\leq \varepsilon, \\
&  \dist\left(\vzero,\nabla_{ \vy}F(\mathbf{x},\mathbf{y})-\partial h(\mathbf{y})-\mathbf{J}^\top_\vp(\mathbf{y})\mathbf{u}\right)\leq \varepsilon,\ 
  \|[\mathbf{p}(\mathbf{y})]_+\|\leq \varepsilon,\ 
 \textstyle \sum_{i=1}^r|\mathbf{p}_i(\mathbf{y})\mathbf{u}_i|\leq \varepsilon.
\end{align*}
\end{definition}
%For each $\vx$, define
%We denote $A\coloneqq\min_{\mathbf{x}}\big\{\zeta(\mathbf{x})+g(\mathbf{x}),  ~\mathrm{s.t.}~  \mathbf{c}(\mathbf{x}) \leq \vzero \big\}$.
%\end{equation}
\begin{definition}[primal $\varepsilon$-optimal point]
For any $\varepsilon\ge 0$, a point $\mathbf x$ is called a primal 
$\varepsilon$-optimal solution of problem \eqref{Basic Formulation} if
%\begin{align*}
$|\zeta(\vx)+g(\vx)-\zeta(\vx^*) - g(\vx^*)|\leq \varepsilon,\,\|[\mathbf{c}(\mathbf{x})]_+\|\leq \varepsilon$, 
%\end{align*}
where $\vx^*$ is a minimizer of $\min_{\mathbf{x}}\big\{\zeta(\mathbf{x})+g(\mathbf{x}),  ~\mathrm{s.t.}~  \mathbf{c}(\mathbf{x}) \leq \vzero \big\}$, with
\begin{equation}\label{newproblem}
    \zeta(\mathbf{x})=\max_{\mathbf{y}}\left\{F(\mathbf{x},\mathbf{y})-h(\mathbf{y}), ~\mathrm{s.t.}~  \mathbf{p}(\mathbf{y}) \leq \vzero \right\}.
\end{equation}
\end{definition}

\section{Technical assumptions and  Proposed Algorithm}
Throughout this paper, we make the following assumptions about the problem \eqref{Basic Formulation}.

% We next present a global reformulation of problem~\eqref{Basic Formulation}, which enables an augmented Lagrangian framework under standard convexity and regularity assumptions.
% This section introduces the structural assumptions and the Slater's condition. The augmented Lagrangian function corresponding to the original problem is then formulated, based on which our main algorithm is developed.
\begin{assumption}%[Structural Conditions]
\label{as1}
The following conditions hold.
\begin{enumerate}[label=(\roman*)]
\item 
$F(\vx,\vy)$ is convex in $\vx$ and concave in $\vy$, and it has $L_F$-Lipschitz continuous gradient, i.e., %for all $\vx_1,\,\vx_2\in \RR^n$, and $\vy_1,\,\vy_2\in \RR^m$,
\begin{equation*}
    \|\nabla F(\vx_1,\vy_1)-\nabla  F(\vx_2,\vy_2)\|\leq L_F\|(\vx_1,\vy_1)-(\vx_2,\vy_2)\|, \forall\, \vx_1,\vx_2\in \RR^n, \forall\, \vy_1,\vy_2\in \RR^m.
\end{equation*}

\item $g$ and $h$ are proper closed convex functions and are Lipschitz continuous. In addition, their domains $\operatorname{dom}(g)$ and $\operatorname{dom}(h)$ are bounded. 

\item Each component of $\vp$ and $\vc$ is convex; their Jacobians are %respectively %constraint functions with Lipschitz continuous Jacobians, with constants 
$L_p$- and $L_c$-Lipschitz continuous, i.e., $\|\vJ_\vc(\vx_1) - \vJ_\vc(\vx_2)\|_F \le L_c\|\vx_1-\vx_2\|$, $\|\vJ_\vp(\vy_1) - \vJ_\vp(\vy_2)\|_F \le L_p\|\vy_1-\vy_2\|,\forall\, \vx_1,\vx_2\in \RR^n, \forall\, \vy_1,\vy_2\in \RR^m$.
\end{enumerate}
\end{assumption}
\begin{assumption}[Slater's Condition]
\label{slatercondition}
There exist ${\vx}_{\mathrm{feas}} \in \operatorname{dom}(g)$ and ${\vy}_{\mathrm{feas}} \in \operatorname{dom}(h)$ such that %$\vc({\vx}_{\mathrm{feas}}) < \vzero, \ \vp({\vy}_{\mathrm{feas}}) < \vzero,$ namely,
$$c_{\mathrm{feas}}:= \min_{1\le i \le q}\big(-{\vc}_i({\vx}_{\mathrm{feas}})\big) > 0,\quad p_{\mathrm{feas}}:= \min_{1\le i \le r}\big(-{\vp}_i({\vy}_{\mathrm{feas}})\big) > 0.$$
\end{assumption}

\subsection{Proposed Algorithmic Framework} Under the assumptions above, we are able to reformulate the problem \eqref{Basic Formulation} by using strong duality to incorporate the functional constraints $\vp(\vy)\le \vzero$ into the new objective.  %holds, and the original minimax problem can be equivalently reformulated as a convex program. 
In particular, let 
\begin{equation}\label{eq:defofQ}
    Q(\vx,\vy,\vu)=F(\mathbf{x}, \mathbf{y}) - h(\mathbf{y}) - \frac{1}{2 \beta_y} \left( \|[\beta_y \mathbf{p}(\mathbf{y}) + \mathbf{u}]_+\|^2 - \|\mathbf{u}\|^2 \right),\ \Psi(\mathbf{x}, \mathbf{u})\coloneqq\max _\mathbf{y}Q(\vx,\vy,\vu),
\end{equation}
where $\vu$ is the Lagrange multiplier, and \(\beta_y > 0\) is the penalty coefficient. Then it holds that for any $\vx\in \RR^n$, $\zeta(\vx) = \min_{ \mathbf{u}} \Psi(\mathbf{x}, \mathbf{u})$, where $\zeta$ is defined in \eqref{newproblem}.
%strong duality implies that for any $\vx\in\RR^n$,
% \begin{align}\label{eq:primal}
%     &\max_\vy\big\{F(\vx,\vy)-h(\vy),\text{ s.t. }\vp(\vy)\leq \vzero \big\} = \min_{ \mathbf{u}} \Psi(\mathbf{x}, \mathbf{u}). %\\
%     % \iff &\min_{ \mathbf{u}}\{ \Psi(\mathbf{x}, \mathbf{u})\coloneqq\max _\mathbf{y}Q(\vx,\vy,\vu)\}, 
% \end{align}
%where 
 %Since $Q(\vx,\vy,\vu)$ is convex in $(\vx,\vu)$ and concave in $\vy$, the maximization over $\vy$ results in a convex function in $(\vx,\vu)$. 
 Hence, problem~\eqref{Basic Formulation} can be equivalently reformulated as
\begin{equation}\label{eq:eqproblem}
    \min_{\vx,\vu}\Psi(\vx,\vu)+g(\vx),\,\text{ s.t. }\vc(\vx)\leq \vzero .
\end{equation}
% The associated Lagrangian dual problem of~\eqref{eq:eqproblem} is given by
% \begin{equation}\label{eq:eqdualproblem}
%     \max_{\vz\geq\vzero}\min_{\vx,\vu}\Psi(\vx,\vu)+g(\vx)+\ip{\vz}{\vc(\vx)}
% \end{equation}
%Embedding the constraint $\vc(\vx)\leq 0$ into the objective gives 

Notice that \eqref{eq:eqproblem} is still a minimax problem because $\Psi$ is defined by a maximization process. Nevertheless, the maximization does not involve a functional constraint and thus is easier to handle. Our algorithm is based on a proximal ALM in~\cite{lu2023iteration} for solving \eqref{eq:eqproblem} but uses a sophisticated subroutine to solve each proximal ALM subproblem due to the inaccessibility of exact gradients. The AL function of problem~\eqref{eq:eqproblem} 
is
\begin{equation}\label{eq:defofL}
L_\beta(\mathbf{x},\mathbf{u},\mathbf{z})= \Psi(\mathbf{x}, \mathbf{u}) + g(\mathbf{x}) + \frac{1}{2\beta} \| [\beta \mathbf{c}(\mathbf{x}) + \mathbf{z}]_+ \|^2-\frac{1}{2\beta}\|\vz\|^2,
\end{equation}
where $\beta>0$ is the penalty parameter, and $\vz$ is the Lagrangian multiplier. %We propose a first-order algorithm based on the proximal augmented Lagrangian framework to solve \eqref{eq:eqproblem}.
%The main $(\vx,\vu)$-update approximately minimizes the proximal augmented Lagrangian subproblem $\vartheta_t$ defined as:
At an iterate $(\vx^t, \vu^t)$, define
\begin{equation}\label{eq:inneriteration}
%(\mathbf{x}^{t+1}, \mathbf{u}^{t+1}) \approx \argmin_{\vx,\vu}
\vartheta_t(\mathbf{x},\mathbf{u},\mathbf{z};\beta,\rho)\coloneqq L_\beta(\mathbf{x},\mathbf{u},\mathbf{z})+\frac{\rho}{2} \left( \|\mathbf{x} - \mathbf{x}^t\|^2 + \|\mathbf{u} - \mathbf{u}^t\|^2 \right),
\end{equation}
where $\rho>0$ is a proximal parameter. Then we give the pseudocode of our framework in %The main algorithm is given below, where $\beta_{x,t}$ denotes the increasing penalty parameter, $\rho_t$ the decreasing proximal regularization parameter, and $\eta_t$ the tolerance parameter for each outer iteration of 
Algorithm~\ref{al:IAL}. Since computing $\nabla \Psi(\vx,\vu)$ needs exact maximization of $Q$ about $\vy$, we generally are only able to obtain an approximate gradient of $\Psi$. Hence, we need an inexact gradient based method to obtain each $(\mathbf{x}^{t+1}, \mathbf{u}^{t+1})$. To achieve the highest efficiency (i.e., lowest complexity), we use an inexact version of Nesterov's accelerated proximal gradient method (iAPG) %. Due to the page limit, its complete analysis is 
given in the note \cite{AcceleratedMethods-Xu2026}. Its pseudocode and %its main 
complexity result are given in the appendix. 

By Theorem~\ref{thm:iter-iAPG}, it is guaranteed that $(\mathbf{x}^{t+1}, \mathbf{u}^{t+1})$ and $\vy^{t+1}$ generated from Algorithm~\ref{al:IAL} satisfy
\begin{align}   
&\dist(\vzero,\partial_{(\vx,\vu)}\vartheta_t(\mathbf{x}^{t+1}, \mathbf{u}^{t+1}, \mathbf{z}^t; \beta_{x,t}, \rho_t))\leq \eta_t, \label{eq:cond-x-u} \\
&\label{eq:y-t+1-iter}
\dist( \vzero,\nabla_\mathbf{y}F(\mathbf{x}^{t+1},\mathbf{y}^{t+1})-\partial h(\mathbf{y}^{t+1})-\mathbf{J}^\top_\vp(\mathbf{y}^{t+1})[\beta_y \mathbf{p}(\mathbf{y}^{t+1})+\mathbf{u}^{t+1}]_+)\leq \eta_t \mu_F.
\end{align}
\begin{algorithm}[H]
\caption{Adaptive inexact Proximal Augmented Lagrangian Method for \eqref{eq:eqproblem}%(Inexact Version)
}
\label{al:IAL}
\begin{algorithmic}[1]
\State \textbf{Input:} $\varepsilon > 0$, $\underline{L}>0$, $\gamma_u>1$, and $\gamma_d>1$; initial point $\mathbf{x}^0\in \dom(g), \mathbf{u}^0\geq\vzero,$ and $ \mathbf{z}^0\ge \vzero$.
\State $\textbf{Set:}$ %$\mathcal{B}_z^0=\max\{1,\|\vz^0\|\}$,
$\beta_{x,t} = \beta_{x,0} \sigma^t$, $\rho_t = \rho_0 \sigma^{-t}$, and $\eta_t = \eta_0 \tau^{-t}$ for each $t\ge0$, where $\tau > \sigma > 1$.
\For{$t = 0, 1, \ldots$}
\State Define $\vv = (\vx,\vu)$ and
\begin{equation}\label{eq:def-psi-phi-t}
\psi_t(\vv)=\Psi(\mathbf{x}, \mathbf{u}) +  \frac{1}{2\beta_{x,t}} \| [\beta_{x,t} \mathbf{c}(\mathbf{x}) + \mathbf{z}^t]_+ \|^2, \ \phi_t(\vv) = g(\mathbf{x}) + \frac{\rho_t}{2} \left( \|\mathbf{x} - \mathbf{x}^t\|^2 + \|\mathbf{u} - \mathbf{u}^t\|^2 \right). %, \ P(\vv) = \phi(\vv)+\psi(\vv).
\end{equation}
\State %Apply Algorithm~\ref{al:iAPG} to~\eqref{eq:inneriteration} to 
Call Algorithm~\ref{al:iNAPG}:  
$(\mathbf{x}^{t+1}, \mathbf{u}^{t+1}) = \mathrm{iAPG}\big(\psi_t,\phi_t; (\vx^t,\vu^t), \underline{L}\sigma^t, \rho_t, \eta_t, \gamma_u,\gamma_d\big)$. 
\State Define
\begin{align*}
    \bar{\psi}_t(\vy)=-F(\vx^{t+1}, \vy) + \frac{1}{2 \beta_y} \|[\beta_y \vp(\vy) + \vu^{t+1}]_+\|^2 -\frac{\mu_F}{2}\|\vy\|^2,\ \bar{\phi}_t(\vy)=\frac{\mu_F}{2}\|\vy\|^2+h(\vy).
\end{align*}
\State Call the exact version of Algorithm~\ref{al:iNAPG}:  $\vy^{t+1}=\mathrm{iAPG}\big(\bar{\psi}_t, \bar{\phi}_t; \vy^t, \underline{L}, \mu_F, \min\{\varepsilon,\eta_t\} \mu_F, \gamma_u, \gamma_d\big)$. 
    \State Let $\mathbf{z}^{t+1} = [\beta_{x,t} \mathbf{c}(\mathbf{x}^{t+1}) + \mathbf{z}^t]_+$ and $\tilde\vu^{t+1} = [\beta_y \vp(\vy^{t+1})+\vu^{t+1}]_+$.
    %\State Set $\mathcal{B}_z^{t+1}=\max\{\|\vz^{t+1}\|,\mathcal{B}_z^t\}$.
    \If{$(\vx^{t+1},\vy^{t+1})$ is an $\varepsilon$-KKT point of \eqref{Basic Formulation} with multipliers $\vz^{t+1}$ and $\tilde\vu^{t+1}$}
        \State \Return $(\mathbf{x}^{t+1}, \vy^{t+1}, \mathbf{z}^{t+1}, \tilde{\vu}^{t+1})$ and \textbf{stop}
    \EndIf
\EndFor
\end{algorithmic}
\end{algorithm}

\section{Complexity Analysis}
In this section, we present the complexity analysis of the proposed method for both the convex-strongly-concave and convex-concave cases. %, we derive the gradient-evaluation complexity of Algorithm~\ref{al:IAL} with respect to $\vF$, $\vc$, and $\vp$, and establish that the resulting point is an $\varepsilon$-KKT point, which can be shown to be %For the convex-strongly-concave case, we further establish that the output point is also 
%a primal $\cO(\varepsilon)$-optimal point. 
Notice that we can apply the result in \cite{lu2023iteration} to obtain a bound of the number of calls to $\nabla \Psi$ and $\prox_{\alpha_x g}$ to produce an $\varepsilon$-KKT point of \eqref{eq:eqproblem}. However, we do not have access to exact gradient $\nabla \Psi$, and in addition, our aim is to obtain an $\varepsilon$-KKT point of \eqref{Basic Formulation} defined in Definition~\ref{def1}. Hence, a new analysis is needed to establish the complexity result and particularly to establish the inner iteration complexity (shown in Theorem~\ref{thm:each iteration complexity}) of calling $\nabla_\vy F$, which turns out  challenging due to the lack of a uniform boundedness of $\vu$ iterates.

\subsection{Convex Strongly-Concave Case}
In this subsection, we analyze Algorithm~\ref{al:IAL} for the convex-strongly-concave case.
\begin{assumption}[Strong Concavity]\label{as2}
    For any $\vx \in \dom(g)$, $F(\vx, \cdot)$ is $\mu_F$-strongly concave with $\mu_F>0$.
\end{assumption}
Under this assumption, the maximizer of $Q$ about $\vy$ is unique. %problem~\eqref{eq:eqproblem}{\color{red} is this cite right? } admits a unique solution. %, which enables the use of Danskin’s theorem to compute the gradient of $\Psi$. 
Specifically, for any $(\vx, \vu)$, let
\begin{equation}\label{eq:solve y}
     \mathbf{y}(\mathbf{x}, \mathbf{u})=\argmax_{\mathbf{y}} Q(\vx,\vy,\vu)
\end{equation}
be the unique maximizer. Consequently,
\begin{equation}\label{eq:findpsi}
    \Psi(\mathbf{x}, \mathbf{u}) = F(\mathbf{x}, \mathbf{y}(\mathbf{x}, \mathbf{u})) - h(\mathbf{y}(\mathbf{x}, \mathbf{u})) - \frac{1}{2 \beta_y} \left( \|[\beta_y \mathbf{p}(\mathbf{y}(\mathbf{x}, \mathbf{u})) + \mathbf{u}]_+\|^2 - \|\mathbf{u}\|^2 \right).
\end{equation}
By Danskin's Theorem \cite{bertsekas2016nonlinear}, $\Psi$ is differentiable, and its gradient is
\begin{equation}\label{eq:thenablaofpsi}
    \nabla \Psi(\mathbf{x},\mathbf{u})=\left[\nabla_{\mathbf{x}} F(\mathbf{x},\mathbf{y}(\mathbf{x},\mathbf{u})),\frac{1}{\beta_y}\left(\mathbf{u} -[\beta_y\,\vp(\mathbf{y}(\mathbf{x},\mathbf{u})) + \mathbf{u}]_+\right)\right].
\end{equation}
\subsubsection{Preliminary Lemmas}
%Before analyzing the iteration complexity of Algorithm~\ref{al:IAL} under the stated assumptions, 
We first introduce some notations and establish several auxiliary lemmas, whose proofs are in the appendix.
By Assumption~\ref{as1}, %the domains of $g$ and $h$ are compact, and $\vc$ and $\vp$ are Lipschitz smooth, thus 
the following defined constants are finite:
\begin{align}\label{sectionnotation}
&\begin{aligned}
&D_x := \max_{\vx_1,\vx_2 \in \dom(g)} \|\vx_1 - \vx_2\|,\, D_y := \max_{\vy_1,\vy_2 \in \dom(h)} \|\vy_1 - \vy_2\|,\\
&D_{c}: = \max_{\vx \in \dom (g)} \|\mathbf{c}(\mathbf{x}) \|\,,\,D_p\coloneqq\max_{\vy \in \dom (h)}\|\vp(\vy)\|,\\
&M_F:= \max_{\substack{\vx\in\dom(g)\\ \vy\in\dom(h)}} \|\nabla F(\vx, \vy)\|, \, M_{p}\coloneqq\max_{\vy\in\dom(h)}\|\vJ_\vp({\vy})\|,\,
M_{c}\coloneqq\max{\vx\in\dom(g)}\|\vJ_\vc({\vx})\|,%\notag
\\
&%\label{boundofFandh}
    \Delta_F=\max_{
\substack{
\vx \in \dom(g) \\
\vy_1, \vy_2 \in \dom(h)
}
} | F(\vx, \vy_1) - F(\vx, \vy_2) |,\quad\Delta_h=\max_{
\vy_1, \vy_2 \in \dom(h)
} | h( \vy_1) - h( \vy_2) |.
\end{aligned}
\end{align}
Notice that except an upper bound on $M_p$, we do not need to know the values of these constants in our algorithm but only use them in our analysis.
%These constants bound the domains and gradients used in later complexity bounds.

% We define
% \begin{equation}\label{eq:defofD}
% D_0\coloneqq\operatorname{dist}\left((\mathbf{x}^0,\mathbf{u}^0),\mathcal{X}^*\right)+\operatorname{dist}(\mathbf{z}^0,\mathcal{Z}^*),
% \end{equation}
% where $\mathcal{X}^*$ is the set of optimal solutions to problem \eqref{eq:eqproblem}, and $\mathcal{Z}^*$ denotes the optimal solution set of the dual problem of~\eqref{eq:eqproblem}.

%Beyond 
With these constants, we bound the $\vu$- and $\vz$-iterates below. %defined above, our analysis will also use bounds on , established in the following lemmas, whose proofs are given in the appendix. %establish bounds on $\vu$ and $\vz$.
\begin{lemma}
\phantomsection
\label{appendixboundofu}
    Under Assumptions~\ref{as1}--\ref{slatercondition}, let $\{\vu^t, \tilde\vu^t\}$ be generated from Algorithm~\ref{al:IAL}. Then %satisfies
    \begin{equation}\label{eq:boundofueq}
        \|\vu^t\|\leq D_u\coloneqq\max\left(\|\vu^0\|,\frac{\Delta_F+\Delta_h}{p_\text{feas}}\right)+\frac{\eta_0}{\rho_0}\frac{\tau}{\tau-\sigma}+\beta_y\eta_0\frac{\tau}{\tau-1},~\forall t\geq0,
    \end{equation} %where $p_{\mathrm{feas}}\coloneqq\min_j(-\vp_j(\mathbf{y}_{\mathrm{feas}}))$, $\vy_{\mathrm{feas}}$ is defined in Assumption~\ref{slatercondition}, and $\Delta_F$ and $\Delta_h$ are defined in~\eqref{boundofFandh}.
%where $p_{\mathrm{feas}}\coloneqq\min_j(-\vp_j(\mathbf{y}_{\mathrm{feas}}))$, 
and  $\|\tilde\vu^t\| \le \beta_y D_p + D_u$ for all $t\ge0$. Moreover, it holds that $\|\vu^{t,*}\|\leq D_u$, $\forall t\geq 0$, where 
\begin{equation}\label{eq:x_u*}
    (\vx^{t,*},\vu^{t,*})=\argmin_{(\vx,\vu)}\vartheta_t(\mathbf{x}, \mathbf{u}, \mathbf{z}^t; \beta_{x,t}, \rho_t),\, \forall t\geq 0.
\end{equation}
Furthermore, for any $\vx\in\dom(g)$, it holds $   \|\hat{\vu}\|\leq D'_u\coloneqq\frac{\Delta_F+\Delta_h}{p_{\mathrm{feas}}}$, where %let $\hat{\vu}$ be defined as %an optimal solution of the Lagrange dual problem associated with~\eqref{newproblem}, i.e.,
\begin{equation*}
  \hat{\vu} \in  \Argmin_{{\vu\geq \vzero}}d(\vu,\vx):= \max_{\vy}\big\{F(\vx,\vy)-h(\vy)-\ip{\vu}{\vp(\vy)}\big\}.
\end{equation*}
\end{lemma}
% \begin{proof}
%     See Appendix~\ref{appendixboundofuproof}.
% \end{proof}
% \begin{lemma}
% \phantomsection
% \label{appendixboundofuhat}
%     Under Assumptions~\ref{as1}-\ref{slatercondition}, 
%Then 
%\begin{equation}\label{eq:def-u'}

%    \end{equation}
%    where $p_{\mathrm{feas}}\coloneqq\min_j(-\vp_j(\mathbf{y}_{\mathrm{feas}}))$. %$\vy_{\mathrm{feas}}$ is defined in Assumption~\ref{slatercondition}.
    %, and $\Delta_F$ and $\Delta_h$ are defined in~\eqref{boundofFandh}.
% \end{lemma}

With $D_u$ given in \eqref{eq:boundofueq}, we further define
\begin{equation}\label{boundofpsiandg}
    \Delta_\Psi=\max_{
\substack{
\vx_1, \vx_2\in \dom(g)\\
\|\vu\|\leq D_u
}
} | \Psi(\vx_1, \vu) - \Psi(\vx_2, \vu) |,\quad\Delta_g=\max_{
\vx_1, \vx_2 \in \dom(g)
} | g( \vx_1) - g( \vx_2) |,
\end{equation}
which are both finite due to compactness of $\dom(g)$.

% \begin{lemma}
% \phantomsection
% \label{appendixboundofzstar}
%     Under Assumptions~\ref{as1}-\ref{slatercondition}, 
%  %   \end{equation}
%     %where $c_{\mathrm{feas}}\coloneqq\min_j(-\vc_j(\mathbf{x}_{\mathrm{feas}}))$.
% \end{lemma}

% where $D_u$ is defined in~\eqref{eq:boundofueq}.
\begin{lemma}
\phantomsection
\label{appendixboundofz}
    Under Assumptions~\ref{as1}-\ref{slatercondition}, for any $\vu$ satisfying $\|\vu\|\leq D_u$, let $\hat{\vz}$ be a multiplier corresponding to an optimal solution of \eqref{eq:eqproblem}. It then holds
%\begin{equation}\label{eq:Dz'}
    $\|\hat{\vz}\|\leq D'_z\coloneqq\frac{\Delta_\Psi+\Delta_g}{c_{\mathrm{feas}}}$. Moreover, let $\{\vz^t\}$ be generated from Algorithm~\ref{al:IAL}. Then 
\begin{equation*}
\|\mathbf{z}^{t}\|\leq D_z\coloneqq D_z'+\sqrt{(D_z')^2+\|\mathbf{z}^{0}\|^2+\beta_{x,0}\rho_0(D_x^2+4D_u^2)+2\beta_{x,0}\eta_0\frac{\tau}{\tau-\sigma}\sqrt{D_x^2+4D_u^2}},\ \forall\, t\ge0.
\end{equation*}
%where $D_z'$ is defined in Lemma~\ref{appendixboundofzstar}. %~\eqref{eq:Dz'}.
%$D_u$ is defined in~\eqref{eq:boundofueq}, 
\end{lemma}

The following lemma characterizes the Lipschitz constant of $\nabla\Psi(\vx,\vu)$.

\begin{lemma}
\phantomsection
\label{pr:Lipc}
    Let $\Psi(\vx,\vu)$ be defined in~\eqref{eq:defofQ}. Under Assumption~\ref{as1}--\ref{as2}, $\nabla\Psi(\vx,\vu)$ is Lipschitz continuous with constant $L_\Psi$, where
\begin{equation}\label{eq:Lip-grad-Psi}
    L_\Psi=\max\left\{\sqrt{L_F^2+\frac{2L_F^4}{\mu_F^2}+\frac{4M_p^2L_F^2}{\mu_F^2}},\sqrt{\frac{2M_p^2L_F^2}{\mu_F^2}+\frac{8}{\beta_y^2}+\frac{4M_p^4}{\mu_F^2}}\right\}.
\end{equation}
\end{lemma}

The next lemma %shows an error 
bounds the distance of $\nabla_{(\vx,\vu)} Q(\vx,\vy,\vu)$ to %an inexact gradient of 
$\nabla\Psi$ if $\vy$ is an approximate maximizer. %by approximately maximizing $Q$ about $\vy$. %choose the tolerance parameter in Algorithm~\ref{al:iNAPG} with exact gradient.
\begin{lemma}
\phantomsection
\label{lemma1}
%Suppose 
Under Assumptions~\ref{as1}--\ref{as2}, let $\delta>0$, $\vx\in\dom(g)$, and $\vu$ be given. If $\vy\in\dom(h)$ satisfies
\begin{equation}\label{eq:lemma353}
    \dist\big( \vzero,\partial_\vy Q(\vx,\vy,\vu)\big)\leq \frac{\delta\mu_F}{(L_F+M_p)},
\end{equation}
then
    $\|\nabla\Psi(\vx,\vu)-\nabla_{(\vx,\vu)} Q(\vx,\vy,\vu)\|\leq\delta.$

\end{lemma}

The next results are well-known, cf. Lemma~1 of~\cite{lu2023iteration} and Lemma~2 of~\cite{xu2021first-ALM}. We state them here for convenience of the readers and our analysis.

\begin{lemma}
\phantomsection
\label{lm:Lipconstant}
 Let $\beta>0$ and define 
 $
        R_\beta(\vx,\vz)=\frac{1}{2\beta}\|[\beta\vc(\vx)+\vz]_+\|^2,\,
        S_\beta(\vy,\vu)=\frac{1}{2\beta}\|[\beta \vp(\vy)+\vu]_+\|^2.
    $
    Then under Assumption~\ref{as1}, the following statements hold
    \begin{itemize}[labelsep=0.5em,leftmargin=1.5em]
        \item[(1)] $R_\beta(\cdot,\vz)$ and $S_\beta(\cdot,\vu)$ are convex and continuously differentiable. %in $\vx$, and  is convex and continuously differentiable in $\vy$. 
        \item[(2)]$\nabla_{\vx}R_\beta(\cdot,\vz)$ and $\nabla_{\vy}S_\beta(\cdot,\vu)$ are Lipschitz continuous with respective Lipschitz constants given by
   \begin{align*}
       L_c(\beta D_c+\|\vz\|)+M_c^2\beta~\text{ and }~
       L_p(\beta D_p+\|\vu\|)+M_p^2\beta, 
   \end{align*}
   where $L_c$, $D_c$, $D_p$, $M_c$, and $M_p$ are defined in~\eqref{sectionnotation}.
    \end{itemize}
\end{lemma}

\subsubsection{Iteration Complexity}
In this subsection, we analyze the iteration complexity of Algorithm~\ref{al:IAL}. We first give its outer iteration complexity.
%We next give a bound on the total number of iterations required by Algorithm~\ref{al:IAL}.

\begin{theorem}[Outer Iteration Complexity]\label{thm:outer-iter}
Under Assumptions~\ref{as1}--\ref{as2}, given $\varepsilon>0$, Algorithm~\ref{al:IAL} must stop and return an $\varepsilon$-KKT point of \eqref{Basic Formulation} within $N$ outer iterations, where
\begin{equation}\label{eq:def-N}
\begin{aligned}
N = & \Bigg\lceil \max \Bigg\{\log_\sigma\frac{2\rho_0(D_x+2D_u)\max\{1, 3\beta_yD_p+D_u\} }{\varepsilon},\ \log_\sigma \frac{2\rho_0 D_z\max\{1,D_z\}}{\varepsilon}, \\
& \hspace{9mm} \log_\tau \frac{2\eta_0(1+\max\{L_p, M_p\})}{\varepsilon},\ \log_\tau \frac{2\eta_0\big((2\beta_yD_p+D_u)(1+M_p) + \beta_y D_p\big)}{\varepsilon}
\Bigg\} \Bigg\rceil.
\end{aligned}
\end{equation}
\end{theorem}

\begin{proof}
From the last two equations in the proof of Theorem 4(i)  in~\cite[pp.1179]{lu2023iteration}, it holds that 
\begin{align}
    \dist\left(\vzero ,\nabla \Psi(\mathbf{x}^{t+1},\mathbf{u}^{t+1})+\begin{bmatrix}
\partial g(\mathbf{x}^{t+1}) \\
\vzero
\end{bmatrix}+\begin{bmatrix}
\mathbf{J}^\top_\mathbf{c}(\mathbf{x}^{t+1})\mathbf{z}^{t+1} \\
\vzero
\end{bmatrix}\right)\leq \rho_t(\|\vx^{t+1}-\vx^t\| + \|\vu^{t+1}-\vu^t\|) + \eta_t, \label{eq:full differential}\\
\dist\big(\vc(\vx^{t+1}),\mathcal{N}_{{\RR^q_+}}(\vz^{t+1})\big)\leq \rho_t\|\vz^{t+1} - \vz^t\|.\label{eq:the312q1}
\end{align}
Let $\bar{\vy}^{t+1}=\vy(\vx^{t+1},\vu^{t+1})$ by the definition in~\eqref{eq:solve y}. Then from \eqref{eq:thenablaofpsi},  the condition in \eqref{eq:full differential}, together with the definition of $D_x$ in \eqref{sectionnotation} and Lemma~\ref{appendixboundofu}, implies 
    \begin{align}
\dist\left(\vzero ,\nabla_\vx F(\mathbf{x}^{t+1},\bar{\mathbf{y}}^{t+1})+\partial g(\mathbf{x}^{t+1})+\mathbf{J}^\top_\mathbf{c}(\mathbf{x}^{t+1})\mathbf{z}^{t+1}\right)\leq \rho_t(D_x+2D_u) + \eta_t,\label{eq:oneofcondition}\\
\|\mathbf{u}^{t+1} -[\beta_y \mathbf{p}(\bar{\vy}^{t+1}) + \mathbf{u}^{t+1}]_+\|\leq\beta_y\big(\rho_t(D_x+2D_u) + \eta_t\big).\label{eq:inequality for bound U and p(y)}
    \end{align}

% \begin{equation}\label{eq:outtierationeq}
%             \frac{D_0 + \sum_{t=0}^{N} \rho_t^{-1} \eta_t}{\rho_N^{-1}} \leq \frac{\varepsilon'}{2}, \quad \eta_N \leq \frac{\varepsilon'}{2},
%     \end{equation}

In addition, it follows from Lemma~\ref{lemma1} and Eqn.~\eqref{eq:lemma354} in its proof and the choice of $\mathbf{y}^{t+1}$ in \eqref{eq:y-t+1-iter} that
    \begin{equation}\label{eq:boundofyy}
        \|\mathbf{y}^{t+1}- \bar{\mathbf{y}}^{t+1}\|\leq \min\{\varepsilon, \eta_t\}.
    \end{equation}
Hence, by the $L_F$-Lipschitz continuity of $\nabla F$, we have from \eqref{eq:oneofcondition} and~\eqref{eq:boundofyy} that
\begin{align}
      \dist\left(\vzero ,\nabla_{ \vx}F(\mathbf{x}^{t+1},\mathbf{y}^{t+1})+\partial g(\mathbf{x}^{t+1})+\mathbf{J}^\top_\mathbf{c}(\mathbf{x}^{t+1})\mathbf{z}^{t+1}\right)\leq \rho_t(D_x+2D_u) + \eta_t+L_F\eta_t.\label{eq:condition1}
\end{align}
Also, define $I_0=\{i:\vz_i^{t+1}=0\}$ and $I_+=\{i:\vz_i^{t+1}>0\}$. Then 
\begin{equation*}
    \dist\big(\vc(\vx^{t+1}),\mathcal{N}_{\RR_q^+}(\vz^{t+1})\big)=\sqrt{\sum_{i\in I_0}([\vc_i(\vx^{t+1})]_+)^2+\sum_{i\in I_+}(\vc_i(\vx^{t+1}))^2}.
\end{equation*}
Hence, it follows from~\eqref{eq:the312q1} and Lemma~\ref{appendixboundofz} that 
\begin{align}
    \|[\mathbf{c}(\mathbf{x}^{t+1})]_+\|\leq\sqrt{\sum_{i\in I_0}([\vc_i(\vx^{t+1})]_+)^2+\sum_{i\in I_+}(\vc_i(\vx^{t+1}))^2}\leq 2\rho_t D_z,\label{eq:condition51}\\
    \sum_{i=1}^q|\mathbf{c}_i(\mathbf{x}^{t+1})\mathbf{z}_i^{t+1}|\leq \sqrt{\sum_{i\in I_+}(\vc_i(\vx^{t+1}))^2}\sqrt{\sum_{i\in I_+}(\vz_i^{t+1})^2}\leq 2\rho_t D_z^2.\label{eq:condition52}
\end{align}
    
    Moreover, define $I_1=\{i:\mathbf{u}_i^{t+1}\geq-\beta_y\mathbf{p}_i(\bar{\vy}^{t+1})\}$ and $I_2=I_1^c=\{i:\mathbf{u}_i^{t+1}<-\beta_y\mathbf{p}_i(\bar{\vy}^{t+1})\}$. Then 
    \begin{align*}
        \mathbf{u}_i^{t+1} -[\beta_y \mathbf{p}_i(\bar{\vy}^{t+1}) + \mathbf{u}_i^{t+1}]_+=\left\{
\begin{array}{ll}
   -\beta_y\mathbf{p}_i(\bar{\vy}^{t+1})& \text{if } i\in I_1, \\[1mm]
\mathbf{u}_i^{t+1}& \text{if } i\in I_2.
\end{array}
\right.
    \end{align*}
Hence, \eqref{eq:inequality for bound U and p(y)} indicates
    \begin{align}
        \sum_{i\in I_1}\beta_y^2\mathbf{p}^2_i(\bar{\vy}^{t+1})+\sum_{i\in I_2}(\mathbf{u}_i^{t+1})^2\leq \beta_y^2 \big(\rho_t(D_x+2D_u) + \eta_t\big)^2,\label{eq:boundofupv1}
    \end{align}    
    which further implies
     \begin{align}   \sum_{i\in I_1}\mathbf{p}^2_i(\bar{\vy}^{t+1})\leq \big(\rho_t(D_x+2D_u) + \eta_t\big)^2,\quad \sum_{i\in I_2}(\mathbf{u}_i^{t+1})^2\leq \beta_y^2 \big(\rho_t(D_x+2D_u) + \eta_t\big)^2.\label{eq:B_1}
    \end{align}
    Let $I_3=\{i:0<\mathbf{p}_i(\bar{\vy}^{t+1})<-\frac{\vu_i^{t+1}}{\beta_y}\}\subseteq I_2$. Then we have
$$    
        \|[\mathbf{p}(\bar{\vy}^{t+1})]_+\|\leq \sqrt{\sum_{i\in I_1}\vp^2_i(\bar{\vy}^{t+1})+\sum_{i\in I_3}\vp^2_i(\bar{\vy}^{t+1})} \leq \sqrt{\sum_{i\in I_1}\vp^2_i(\bar{\vy}^{t+1})+\sum_{i\in I_3}\frac{(\vu_i^{t+1})^2}{\beta_y^2}},
$$    
which together with \eqref{eq:boundofupv1} gives  
\begin{equation}\label{eq:condition3}
\|[\mathbf{p}(\bar{\vy}^{t+1})]_+\|\leq \rho_t(D_x+2D_u) + \eta_t.
\end{equation}
    Consequently, by the $M_p$-Lipschitz continuity of $\vp$ and \eqref{eq:boundofyy}, we obtain
\begin{equation}\label{eq:boundofyfix}
    \|[\mathbf{p}({\vy^{t+1}})]_+\|\leq \rho_t(D_x+2D_u) + \eta_t + M_p \eta_t.
\end{equation}

Let $\bar{\vu}^{t+1}=[\beta_y\vp(\bar{\vy}^{t+1}) + \vu^{t+1}]_+$. Then by the definition of $\tilde{\vu}^{t+1}$ in Line~8 of Algorithm~\ref{al:IAL}, it follows from \eqref{eq:inequality for bound U and p(y)}, \eqref{eq:boundofyy}, and the $M_p$-Lipschitz continuity of $\vp$ that
\begin{equation}\label{eq:thm312eq2}
    \|\vu^{t+1}-\tilde{\vu}^{t+1}\|\leq\|\vu^{t+1}-\bar{\vu}^{t+1}\|+\|\bar{\vu}^{t+1}-\tilde{\vu}^{t+1}\|\leq \beta_y\big(\rho_t(D_x+2D_u) + \eta_t\big) +  \beta_yM_p\eta_t.
\end{equation}
Also, by \eqref{eq:B_1}, \eqref{eq:thm312eq2}, and the $M_p$-Lipschitz continuity of $\vp$, it follows
\begin{align}
    \sqrt{\sum_{i\in I_1}\mathbf{p}^2_i({\vy^{t+1}})}\leq \sqrt{\sum_{i\in I_1}\mathbf{p}^2_i(\bar{\vy}^{t+1})}+\|\vp(\vy^{t+1})-\vp(\bar{\vy}^{t+1})\|\leq \rho_t(D_x+2D_u) + \eta_t + M_p \eta_t,\label{eq:thm312eq3}\\
    \sqrt{\sum_{i\in I_2}(\tilde{\mathbf{u}}_i^{t+1})^2}\leq \sqrt{\sum_{i\in I_2}(\mathbf{u}_i^{t+1})^2}+\|\tilde{\vu}^{t+1}-\vu^{t+1}\|\leq 2\beta_y\big(\rho_t(D_x+2D_u) + \eta_t\big) +  \beta_yM_p\eta_t.\label{eq:thm312eq4}
\end{align}
 Therefore, using~\eqref{eq:thm312eq3} and~\eqref{eq:thm312eq4} and noticing $\|\tilde{\vu}^{t+1}\| \le \beta_yD_p+D_u$, we obtain
    \begin{align}
        &~\sum_{i=1}^r|\vp_i({\vy^{t+1}})\tilde{\vu}^{t+1}_i|=\sum_{i\in I_1}|\vp_i({\vy^{t+1}})\tilde{\vu}_i^{t+1}|+\sum_{i\in I_2}|\vp_i({\vy^{t+1}})\tilde{\vu}_i^{t+1}|\notag\\
        \leq& ~ \|\tilde{\vu}^{t+1}\|\sqrt{\sum_{i\in I_1}\mathbf{p}_i^2({\vy^{t+1}})} + \|\mathbf{p}({\vy^{t+1}})\| \sqrt{\sum_{i\in I_2}(\tilde{\vu}_i^{t+1})^2}\notag\\
        \leq& ~(\beta_yD_p+D_u) \big(\rho_t(D_x+2D_u) + \eta_t + M_p \eta_t\big)+D_p\left(2\beta_y\big(\rho_t(D_x+2D_u) + \eta_t\big) +  \beta_yM_p\eta_t\right)\notag\\
        = & ~ (3\beta_yD_p+D_u) (D_x+2D_u) \rho_t + \big((2\beta_yD_p+D_u)(1+M_p) + \beta_y D_p\big) \eta_t.\label{eq:condition2}
    \end{align}

Now by Definition~\ref{def1}, we conclude that $(\vx^{t+1},\vy^{t+1})$ is an $\varepsilon$-KKT point of \eqref{Basic Formulation} from \eqref{eq:condition1}, \eqref{eq:condition51}, \eqref{eq:condition52}, \eqref{eq:y-t+1-iter}, \eqref{eq:boundofyfix}, and \eqref{eq:condition2} when $t\ge N$ with $N$ satisfying the conditions in \eqref{eq:def-N}.
\end{proof}

To establish the total complexity of Algorithm~\ref{al:IAL}, we need to estimate the complexity for finding $(\vx^{t+1}, \vu^{t+1})$ for each $t\le N$. We can directly use Theorem~\ref{thm:iter-iAPG} %in the appendix 
to bound the number of evaluations of inexact gradient of $\Psi$. However, unlike Lemma~\ref{appendixboundofu}, we do not have a universal bound on intermediate points about $(\vx,\vu)$, say $(\hat\vx,\hat\vu)$, that are generated during the iterations of Algorithm~\ref{al:iNAPG}; we must be very meticulous in bounding the overall inner iterations for approximately maximizing $Q(\hat\vx,\hat\vu, \vy)$ about $\vy$. 
\begin{theorem}[Complexity for each outer iteration]\label{thm:each iteration complexity}
Under Assumptions~\ref{as1}--\ref{as2}, for each $t\ge0$, to produce $(\vx^{t+1}, \vu^{t+1})$ in Line~5 of Algorithm~\ref{al:IAL} such that the condition in \eqref{eq:cond-x-u} holds, it suffices to take $K_{x,t}$ calls to $\nabla_\vx F$, $J_\vc$, and $\prox_{\alpha_x g}$, and $K_{y,t}$ calls to $\nabla_\vy F$, $J_\vp$, and $\prox_{\alpha_y h}$, where
\begin{align}
    K_{x,t} \le & ~ 2 n_{LS,x,t}\left(T_{\eta_t}+1\right), \label{eq:bound-k-x-t}\\
        K_{y,t}\leq& ~4n_{LS,x,t}n_{LS,y,t}\Bigg(\Bigg\{ (T_{\eta_t}+1)\left\{\sqrt{\frac{8\gamma_u}{\mu_F}}\left(\sqrt{L_F+L_p\beta_{y}D_p+M_p^2\beta_y+\mu_F+L_pG_t}\right)+\frac{1}{\log 2}\right\} \label{eq:bound-k-y-t}\\
    &+\sqrt{\frac{8\gamma_u}{\mu_F}}\left(1+\sqrt{\frac{\rho_t}{\underline{L}\sigma^t}}+\frac{2}{\underline{L}\sigma^t\beta_y}\right){\color{black}\left(\frac{4E_{0,t}\gamma_uL_{\psi,t}}{\rho_t}\right)^{\tfrac{1}{4}}\left(2\sqrt{T_{\eta_t}}\right)}\Bigg\}\notag\\
    &\cdot\log 8\sqrt{\frac{\gamma_u L_{\bar{\psi}_{t,1}} D_y^2}{2\mu_F}}\frac{(L_F+M_p)(1+\gamma_u) L_{\bar{\psi}_{t,1}}}{\omega_{T_{\eta_t}} \mu_F}+{\color{black}3}(T_{\eta_t}+1)\Bigg). \notag
\end{align}
Here, the involved constants are defined as follows for all $t\ge0$,
\begin{align}
&n_{LS,x,t} = \left\lceil \log_{\gamma_u}\frac{L_{\psi,t}}{\underline{L}\sigma^t}\right\rceil+1, \ n_{LS,y,t} = \left\lceil \log_{\gamma_u}\frac{L_{\bar{\psi}_{t,1}}}{\underline{L}}\right\rceil+1, \label{eq:n-LS-x-y-t}\\
&L_{\psi,t}=L_{\Psi}+ L_c D_z + \beta_{x,0}\sigma^t(L_cD_c+M_c^2),\ L_{\bar{\psi}_{t,1}}=L_F+L_p(\beta_{y}D_p+B_{t,1})+M_p^2\beta_y+\mu_F, \label{eq:def-L-psi-Lbar-psi}\\
&B_{t,1}=\sqrt{\frac{E_{0,t}\gamma_uL_{\psi,t}}{\rho_t}}\left(1+\sqrt{\frac{\rho_t}{\underline{L}\sigma^t}}+\frac{2}{\underline{L}\sigma^t\beta_y}\right)+G_t,\\
&G_t=\frac{2}{\underline{L}\sigma^t }\left(M_F + M_c(\beta_{x,t} D_c + D_z) +D_p + \frac{D_u}{\beta_y}\right)+\sqrt{\frac{1}{\underline{L}\sigma^t}}\left(\sqrt{2\Delta_g}+\sqrt{\rho_t}D_x+2\sqrt{\rho_t}D_u\right)+D_u, \label{eq:G-t}\\
&    \omega_{T_{\eta_t}}= \min\left\{\frac{\eta_t}{8},\, \frac{\rho_t \eta_t}{(T_{\eta_t}+1)^2\cdot\max\{1, A_{T_{\eta_t}+1}\}}\right\}, \label{eq:iter-complexity-K-eta} \\
&    T_{\eta_t}\leq \left\lceil \left(\sqrt{\frac{8\gamma_u L_{\psi,t}}{\rho_t}} + \frac{1}{\log 2}\right) \log 8\sqrt{\frac{\gamma_u L_{\psi,t} E_{0,t}}{\rho_t}}\frac{(1+\gamma_u) L_{\psi,t}}{\eta_t}\right\rceil + 1,\label{eq:iter-T-eta}\\
&    A_{T_{\eta_t}+1}\leq\left(1 + \frac{2\rho_t}{\underline{L}\sigma^t} + 2\sqrt{\frac{2\rho_t}{\underline{L}\sigma^t}}\right)^2 \left(\frac{128 E_{0,t}}{\rho_t} \frac{(1+\gamma_u)^2 L_{\psi,t}^2}{\eta_t^2} + \frac{2 + \sqrt{\frac{\underline{L}\sigma^t}{8\rho_t}}}{2\rho_t + 2\sqrt{2\rho_t \underline{L}\sigma^t}}\right),\\
&E_{0,t}=D_x^{2}+2D_u^{2}
 + 30\rho_t\eta_t R_t + \frac{8\rho_t^2 \eta_t^2}{\underline{L}\sigma^t},\label{eq:def-E-0-t}\\ 
&R_t =  \max\Bigg\{2D_x^{2}+8D_u^{2} + \sqrt{\frac{64  \rho_t^2 \eta_t^2 }{\underline{L}\sigma^t}},\; 240\rho_t\eta_t, \; 120 \eta_t\gamma_u L_{\psi,t}, \label{eq:def-R-t}\\
&\qquad\qquad\qquad \frac{\sqrt{2\gamma_u L_{\psi,t}}(D_x^{2}+4D_u^{2})}{\sqrt{\rho_t }} + 16\eta_t \gamma_u L_{\psi,t} + \sqrt{\frac{32\eta_t^2 \rho_t \gamma_u L_{\psi,t}}{\underline{L}\sigma^t}}\Bigg\}. \nonumber
\end{align}
% The definition of $E_0$ will be given at the beginning of the proof, immediately after the notation is introduced.
\end{theorem}
\begin{proof}
For ease of notation, in this proof, we denote 
\begin{equation}\label{eq:temp-notation}
\begin{aligned}
&\vv = (\vx,\vu),\ \beta=\beta_{x,t},\ \rho = \rho_t,\ \eta = \eta_t,\ \vz = \vz^t,\ \vv^0 = (\vx^t, \vu^t),\ \psi(\vv)=\psi_t(\vv),\ \phi(\vv) = \phi_t(\vv), \\
&\vv^*=(\vx^{t,*},\vu^{t,*})=\argmin_{\vv}\psi(\vv)+\phi(\vv),
\end{aligned}
\end{equation}
where $\psi$ and $\phi$ are defined in \eqref{eq:def-psi-phi-t}.
% \begin{equation}\label{eq:def-psi-phi}
% \psi(\vv)=\Psi(\mathbf{x}, \mathbf{u}) +  \frac{1}{2\beta} \| [\beta \mathbf{c}(\mathbf{x}) + \mathbf{z}]_+ \|^2, \ \phi(\vv) = g(\mathbf{x}) + \frac{\rho}{2} \left( \|\mathbf{x} - \mathbf{x}^t\|^2 + \|\mathbf{u} - \mathbf{u}^t\|^2 \right). %, \ P(\vv) = \phi(\vv)+\psi(\vv).
% \end{equation}
Then by Lemma~\ref{lm:Lipconstant}, $\nabla\psi$ is Lipschitz continuous with constant $$L_\psi=L_{\Psi}+L_c(\beta D_c+ \|\vz\|)+M_c^2\beta,$$ and $\phi$ is $\rho$-strongly convex. With these notations, we directly have from Theorem~\ref{thm:iter-iAPG} that we can obtain $(\vx^{t+1},\vu^{t+1})$ in Line~5 of Algorithm~\ref{al:IAL} as an $\eta$-stationary solution of $\psi + \phi$ within {\color{black}$T_\eta+1$} iAPG iterations, where $T_\eta=T_{\eta_t}$ satisfies the bound in \eqref{eq:iter-T-eta}. 
% \begin{equation}
% T_\eta \le  \left\lceil \left(\sqrt{\frac{8\gamma_u L_\psi}{\rho}} + \frac{1}{\log 2}\right) \log 8\sqrt{\frac{\gamma_u L_\psi E_{0}}{\rho}}\frac{(1+\gamma_u) L_\psi}{\eta}\right\rceil + 1,    
% \end{equation}
% with 
% $E_0=\frac{1}{2}\|\vv^{*} - \vv^0\|^{2}
%  + 30\rho\eta R + \frac{8\rho^2 \eta^2}{\underline{L}\sigma^t}$ and
% \begin{align*}
% R = & \max\Bigg\{2\|\vv^{*} - \vv^0\| + \sqrt{\frac{64  \rho^2 \eta^2 }{\underline{L}\sigma^t}},\; 240\rho\eta, \; 120 \eta\gamma_u L_\psi, \nonumber\\
% &\qquad\qquad\qquad \frac{\sqrt{2\gamma_u L_\psi}\|\vv^{*} - \vv^{0}\|}{\sqrt{\rho }} + 16\eta \gamma_u L_\psi + \sqrt{\frac{32\eta^2 \rho \gamma_u L_\psi}{\underline{L}\sigma^t}}\Bigg\}.
% \end{align*}
In addition, the total number, $K_{x,t}$, of calls to inexact gradient of $\psi$ and proximal mapping of $\alpha\phi$ satisfies 
$K_{x,t} \le 2{\color{black}\left(T_\eta+1\right)}\left(\left\lceil \log_{\gamma_u}\frac{L_\psi}{\underline{L}\sigma^t}\right\rceil+1\right),$ which shows the bound in \eqref{eq:bound-k-x-t}.

For each call to an inexact gradient of $\psi$, we need to approximately solve the inner maximization problem about $\vy$. From Algorithm~\ref{al:iNAPG}, we need to compute $\tilde\nabla \psi(\vb)$ and $\tilde\nabla \psi(\vd)$ within the while-loop at each iAPG iteration $k \le T_\eta$, where $\vb$ is a convex combination of $\vv^k$ and $\bm{\pi}^k$ and $\vd = \prox_{\phi/L}\big(\vb - \frac{1}{L}\tilde\nabla \psi(\vb)\big)$. By Theorem~\ref{thm:iter-iAPG}, \eqref{sectionnotation}, and Lemma~\ref{appendixboundofu}, we have  $\|\vv^k - \vv^*\| \le \sqrt{\frac{2E_0}{\rho A_k}}$, $\|\bm{\pi}^k - \vv^*\| \le \sqrt{\frac{2E_0}{\rho A_k}}$, for all $k\ge1$, where $E_0=E_{0,t}$ given in \eqref{eq:def-E-0-t}, and 
\begin{align}
&A_k \ge \max\left\{\frac{k^2}{2\gamma_u L_\psi},\ \frac{2}{\gamma_u L_\psi}\left(1+\sqrt{\frac{\rho}{2\gamma_u L_\psi}}\right)^{2(k-1)}\right\}, \, \forall\, k\ge1, \label{eq:low-bd-A-k} \\
&A_{k} \le \left(1 + \frac{2\rho}{\underline{L}\sigma^t} + 2\sqrt{\frac{2\rho}{\underline{L}\sigma^t}}\right)^2 \left(\frac{128 E_{0}}{\rho} \frac{(1+\gamma_u)^2 L_\psi^2}{\eta^2} + \frac{2 + \sqrt{\frac{\underline{L}\sigma^t}{8\rho}}}{2\rho + 2\sqrt{2\rho \underline{L}\sigma^t}}\right), \forall\, 0\le k  \le T_\eta+1. \label{eq:up-bd-cond-Ak-k}
\end{align}
Hence, for each $k$, 
\begin{equation}\label{eq:bd-b-vv*}
\|\vb - \vv^*\| \le \sqrt{\frac{2E_0}{\rho A_k}},\ \forall\, k\ge1.    
\end{equation} 

To bound $\vd$, we have from the definition of $\vd$ and the $\rho$-strong convexity of $\phi$ that
$$\phi(\vb) \ge \langle \tilde\nabla \psi(\vb), \vd - \vb\rangle + \frac{L+\rho}{2}\|\vd-\vb\|^2 + \phi(\vd).$$
By the definition of $\phi$, $D_x$, and $\Delta_g$ in \eqref{eq:temp-notation}, \eqref{sectionnotation}, and \eqref{boundofpsiandg}, it then follows that
\begin{equation}\label{eq:ineq-ud-ub}  
\Delta_g + \frac{\rho}{2} D_x^2 + \frac{\rho}{2}\|\vu_\vb - \vu^t\|^2 \ge \langle \tilde\nabla \psi(\vb), \vd - \vb\rangle + \frac{L+\rho}{2}\|\vd-\vb\|^2 + \frac{\rho}{2}\|\vu_\vd - \vu^t\|^2, 
\end{equation}
where we denote $\vb = (\vx_\vb, \vu_\vb)$ and $\vd=(\vx_\vd, \vu_\vd)$. 

By \eqref{eq:thenablaofpsi} and Lemma~\ref{lemma1}, if $\tilde\vy$ is an approximate maximizer of $Q(\vx, \cdot, \vu)$, then we can use $$\left[\nabla_{\mathbf{x}} F(\mathbf{x},\tilde\vy) + J_\vc(\vx)^\top[\beta \mathbf{c}(\mathbf{x}) + \mathbf{z}]_+,\tfrac{1}{\beta_y}(\mathbf{u} -[\beta_y\,\mathbf{p}(\tilde\vy) + \mathbf{u}]_+)\right]$$ as an approximation of $\nabla \psi(\vv)$. Hence, we choose
\begin{align*}
&\tilde\nabla \psi(\vb) = \left[\nabla_{\mathbf{x}} F(\vx_\vb,\tilde\vy_\vb) + J_\vc(\vx_\vb)^\top[\beta \mathbf{c}(\mathbf{x}_\vb) + \mathbf{z}]_+,\tfrac{1}{\beta_y}(\mathbf{u}_\vb -[\beta_y\,\mathbf{p}(\tilde\vy_\vb) + \mathbf{u}_\vb]_+)\right],%\\
%&\tilde\nabla \psi(\vd) = \left[\nabla_{\mathbf{x}} F(\vx_\vd,\tilde\vy_\vd) + J_\vc(\vx_\vd)^\top[\beta \mathbf{c}(\mathbf{x}_\vd) + \mathbf{z}]_+,\tfrac{1}{\beta_y}(\mathbf{u}_\vd -[\beta_y\,\mathbf{p}(\tilde\vy_\vd) + \mathbf{u}_\vd]_+)\right],
\end{align*}
where $\tilde\vy_\vb$ %and $\tilde\vy_\vd$ are 
is a sufficiently accurate solution to $\max_\vy Q(\vx_\vb, \vy, \vu_\vb)$ %and $\max_\vy Q(\vx_\vd, \vy, \vu_\vd)$ 
so that $\|\tilde\nabla \psi(\vb) - \nabla \psi(\vb)\| \le\omega_k$ as %and $\tilde\nabla \psi(\vd)$ 
 required in Algorithm~\ref{al:iNAPG}. 
Now by the definition of $M_F$ and $M_c$ in \eqref{sectionnotation}, we have
\begin{align*}
\|\tilde\nabla_\vx \psi(\vb)\| \le & ~ M_F + M_c(\beta D_c + D_z).
\end{align*}
In addition, by Lemma~\ref{appendixboundofu} and the fact $|a - [b+a]_+| \le |a| + |b|$ for any $a, b\in \RR$, it follows
\begin{align*}
&~\|\tilde\nabla_\vu \psi(\vb)\| \le \tfrac{1}{\beta_y}\left(\|\mathbf{u}_\vb\| + \|\beta_y\,\mathbf{p}(\tilde\vy_\vb)\| \right) \le \tfrac{1}{\beta_y}\left(\|\mathbf{u}_\vb - \vu^{t,*}\| + \|\vu^{t,*}\| + \|\beta_y\,\mathbf{p}(\tilde\vy_\vb)\| \right).
\end{align*}
For $k=0$, we use $\vu_\vb = \vu^t$ to have $\|\vu_\vb\|\le D_u$, and for $k\ge1$, we use \eqref{eq:bd-b-vv*} to have $\|\mathbf{u}_\vb - \vu^{t,*}\|\le \sqrt{\frac{2E_0}{\rho A_k}}$.  
Consequently, by $\|\tilde{\nabla} \psi(\vb)\| \le \|\tilde{\nabla}_\vx \psi(\vb)\| + \|\tilde{\nabla}_\vu \psi(\vb)\|$, we obtain
\begin{equation}\label{eq:def-B-b-t-k}
\|\tilde{\nabla} \psi(\vb)\| \le C_{t,k}:=
\left\{
\begin{array}{ll}
M_F + M_c(\beta D_c + D_z) + D_p+\frac{D_u}{\beta_y}, &\text{ if }k=0,\\
M_F + M_c(\beta D_c + D_z) + D_p + \frac{1}{\beta_y}\left(\sqrt{\frac{2E_0}{\rho A_k}} + D_u\right), & \text{ if }k\ge1.
\end{array}
\right.
\end{equation} 

Moreover, noticing $\langle \tilde\nabla \psi(\vb), \vd - \vb\rangle\ge -\|\tilde\nabla \psi(\vb)\|\|\vd - \vb\|$ and solving \eqref{eq:ineq-ud-ub} for $\|\vd-\vb\|$ gives
\begin{align}\label{eq:bound-diff-vd-vb}
\|\vd-\vb\| \le \frac{2}{L+\rho}\|\tilde{\nabla} \psi(\vb)\|+\sqrt{\frac{1}{L+\rho}\left(2\Delta_g+\rho D_x^2+\rho \|\vu_\vb - \vu^t\|^2\right)}\notag\\
\le \frac{2}{\underline{L}\sigma^t}\|\tilde{\nabla} \psi(\vb)\|+\sqrt{\frac{1}{\underline{L}\sigma^t}\left(2\Delta_g+\rho D_x^2+\rho \|\vu_\vb - \vu^t\|^2\right)},   
\end{align}
where the second inequality follows from $L\ge \underline{L}\sigma^t$. For $k=0$, $\vu_\vb = \vu^t$, and we bound $\|\vu_\vd\|\leq\|\vu_\vd-\vu_\vb\|+ \|\vu_\vb\| \le \|\vd-\vb\| + D_u$; for $k\ge1$, $\|\vu_\vb - \vu^t\| \le \|\mathbf{u}_\vb - \vu^{t,*}\| + \|\vu^{t,*} - \vu^t\| \le \sqrt{\frac{2E_0}{\rho A_k}}+2D_u$, and we bound $\|\vu_\vd\|\leq\|\vu_\vd-\vu_\vb\|+ \|\vu_\vb-\vu^{t,*}\|+\|\vu^{t,*}\|\le  \|\vd-\vb\| + \sqrt{\frac{2E_0}{\rho A_k}}+D_u$. Hence, by \eqref{eq:def-B-b-t-k} and \eqref{eq:bound-diff-vd-vb}, we obtain
\begin{align}\label{eq:bd-ud}
    \|\vu_\vd\|\leq B_{t,k}:= 
    \left\{
    \begin{array}{ll}
    \frac{2C_{t,k}}{\underline{L}\sigma^t} + \sqrt{\frac{1}{\underline{L}\sigma^t}\left(2\Delta_g+\rho D_x^2\right)} + D_u, &\text{ if }k=0,\\
    \frac{2C_{t,k}}{\underline{L}\sigma^t} + \sqrt{\frac{1}{\underline{L}\sigma^t}\left(2\Delta_g+\rho D_x^2+\rho\left(\sqrt{\frac{2E_0}{\rho A_k}}+2D_u\right)^2\right)} + \sqrt{\frac{2E_0}{\rho A_k}} + D_u, &\text{ if }k\ge1.
    \end{array}
    \right.
\end{align}
From \eqref{eq:bd-b-vv*} and $\vu_\vb = \vu^t$ when ${\color{black}k=0}$, it holds clearly $\|\vu_\vb\| \le B_{t,k}, \forall\, k\ge0$ as well.

With the bound on $\vu_\vb$ and $\vu_\vd$, we are now ready to estimate the complexity to obtain $\tilde\nabla\psi(\vb)$ and $\tilde\nabla\psi(\vd)$. As {\color{black}they} share the same bound, we only need to show the complexity for one, and we work on the latter. Let us fix $k$, and let
%By Lemma~\ref{lemma1}, to obtain $\tilde{\vy}$ such that $\|\tilde\nabla \psi(\vv) - \nabla \psi(\vv)\| \le\omega_k$, it suffices to compute a $\frac{\omega_k \mu_F}{L_F+M_p}$-stationary point of $\max_\vy Q(\vx,\vy,\vu)$, which can be done by the exact version of Algorithm~\ref{al:iNAPG}. Let $\vy^*=\argmax_\vy Q(\vx,\vy,\vu)$ and let
\begin{align*}
    \bar{\psi}(\vy)=-F(\vx_\vd, \vy) + \frac{1}{2 \beta_y} \left( \|[\beta_y \vp(\vy) + \vu_\vd]_+\|^2 - \|\vu_\vd\|^2 \right)-\frac{\mu_F}{2}\|\vy\|^2,\ \bar{\phi}(\vy)=\frac{\mu_F}{2}\|\vy\|^2+h(\vy).
\end{align*}
By Lemma~\ref{lm:Lipconstant} and \eqref{eq:bd-ud}, it follows that
$\nabla\bar{\psi}$ is Lipschitz continuous with constant
\begin{equation}\label{eq:L-bar-psi-t-k}
   L_{\bar{\psi}_{t,k}}=L_F+L_p(\beta_{y}D_p+B_{t,k})+M_p^2\beta_y+\mu_F.
\end{equation}
%gradient and denote the Lipschitz constant as $L_{\bar{\psi}}$ depending on $\|\vu\|$, and 
Clearly, $\bar{\phi}$ is $\mu_F$-strongly convex. By Lemma~\ref{lemma1}, to ensure $\|\tilde\nabla \psi(\vd) - \nabla \psi(\vd)\| \le\omega_k$, it suffices to compute a $\frac{\omega_k \mu_F}{L_F+M_p}$-stationary point of $\min_\vy -Q(\vx_\vd,\vy,\vu_\vd) = \bar{\psi}(\vy) + \bar{\phi}(\vy)$. %{\color{red}it suffices to compute a $\frac{\omega_k \mu_F}{L_F+M_p}$-stationary point of $\max_\vy Q(\vx_\vd,\vy,\vu_\vd) = -\bar{\psi}(\vy) - \bar{\phi}(\vy)$}
Hence, call the exact version of Algorithm~\ref{al:iNAPG} by $\tilde{\vy} = \mathrm{iAPG}\big(\bar{\psi}, \bar{\phi}; \vy^t, \underline{L}, \mu_F, \frac{\omega_k \mu_F}{L_F+M_p}, \gamma_u, \gamma_d\big)$, which will take $K_{y,t,k}$ calls to $\nabla \bar{\psi}$ and proximal mapping of $\bar{\phi}$. By Theorem~\ref{thm:iter-iAPG}, it holds
%
%With the bounds on $\vu_\vb$ and $\vu_\vd$ in~\eqref{eq:bound-d}, Lemma~\ref{lm:Lipconstant} implies that, for both $(\vx,\vu)=(\vx_\vb,\vu_\vb)$ and $(\vx_\vd,\vu_\vd)$, the gradient of $\bar{\psi}$ is Lipschitz continuous. We denote by $L_{\bar{\psi}_{t,k}}$ a common Lipschitz constant, given by
%
%Let $D_y$ be defined in~\ref{sectionnotation}. Applying Theorem~\ref{thm:iter-iAPG} with $L_{\min}=\underline{L}$ and using $L_{\bar{\psi}_{t,k}}\le L_{\bar{\psi}_{t,1}}$, $\omega_{k}\geq \omega_{T_\eta}$, $\forall 0\leq k\leq T_\eta$, to crudely bound the logarithmic factor, we obtain that the total number $K_{y,t,k}$ of gradient and proximal-mapping calls needed to compute such a point satisfies
\begin{equation*}
    K_{y,t,k} \le 2 n_{LS,y,t}\left(\left\lceil \left(\sqrt{\frac{8\gamma_u L_{\bar{\psi}_{t,k}}}{\mu_F}} + \frac{1}{\log 2}\right) \log 8\sqrt{\frac{\gamma_u L_{\bar{\psi}_{t,1}} D_y^2}{2\mu_F}}\frac{(L_F+M_p)(1+\gamma_u) L_{\bar{\psi}_{t,1}}}{\omega_{T_{\eta}} \mu_F}\right\rceil + {\color{black}2}\right),
\end{equation*}
where we have used the definition of $D_y$ in~\ref{sectionnotation}, $L_{\bar{\psi}_{t,k}}\le L_{\bar{\psi}_{t,1}}$, and $\omega_k {\color{black}\ge} \omega_{T_\eta}$ for $0\le k\le T_\eta$.

To bound the total number, $K_{y,t}$, calls to $\nabla_\vy F$ and $J_\vp$, we need to bound $\sum_{k=0}^{T_\eta}2n_{LS,x,t}K_{y,t,k}$ where $n_{LS,x,t}$ is the maximum number of line-search steps in each while-loop of Algorithm~\ref{al:iNAPG} called in Line~5 of Algorithm~\ref{al:IAL}. By the definition of $G_t$ in \eqref{eq:G-t} and recalling $\beta=\beta_{x,t}, \rho = \rho_t$, we obtain  
\begin{align*}
    B_{t,k}\le 
    \left\{\begin{array}{ll}
    G_t, &\text{ if }k=0, \\
    \sqrt{\frac{2E_0}{\rho A_k}}\left(1+\sqrt{\frac{\rho}{\underline{L}\sigma^t}}+\frac{2}{\underline{L}\sigma^t\beta_y}\right)+G_t,&\text{ if }k\ge1. 
    \end{array}
    \right.
\end{align*}
Thus from \eqref{eq:L-bar-psi-t-k}, it follows 
\begin{equation}\label{eq:bd-sum-Lpsi-mu}
{\color{black}\begin{aligned}
    \sum_{k=0}^{T_\eta}\sqrt{\frac{8\gamma_u L_{\bar{\psi}_{t,k}}}{\mu_F}}\leq& \sqrt{\frac{8\gamma_u}{\mu_F}}\sum_{k=0}^{T_\eta}\left(\sqrt{L_F+L_p\beta_{y}D_p+M_p^2\beta_y+\mu_F+L_pG_t}\right)\\
    &+\sqrt{\frac{8\gamma_u}{\mu_F}}\sqrt{\left(1+\sqrt{\frac{\rho}{\underline{L}\sigma^t}}+\frac{2}{\underline{L}\sigma^t\beta_y}\right)}\sum_{k=1}^{T_\eta}\left(\frac{2E_0}{\rho A_k}\right)^{\tfrac{1}{4}}.
\end{aligned}}
\end{equation}
By~\eqref{eq:low-bd-A-k}, we obtain
\begin{equation}\label{eq:bd-sum-E0-rhpAk}
{\color{black}    \sum_{k=1}^{T_\eta}\left(\frac{2E_0}{\rho A_k}\right)^{\tfrac{1}{4}}\leq\left(\frac{4E_0\gamma_uL_\psi}{\rho}\right)^{\tfrac{1}{4}}\sum_{k=1}^{T_\eta} \frac{1}{\sqrt{k}}\leq \left(\frac{4E_0\gamma_uL_\psi}{\rho}\right)^{\tfrac{1}{4}}\left(2\sqrt{T_\eta}\right) =: H_t.}
\end{equation}
%Then, the total number, $K_{y,t}$, of calls to gradient of $\bar{\psi}$ and proximal mapping of $\alpha\bar{\phi}$ satisfies
Therefore, by \eqref{eq:bd-sum-Lpsi-mu} and \eqref{eq:bd-sum-E0-rhpAk}, it follows
\begin{align*}
    \sum_{k=0}^{T_\eta}K_{y,t,k}
    \leq&2n_{LS,y,t}\Bigg(\Bigg\{ (T_\eta+1)\left\{\sqrt{\frac{8\gamma_u}{\mu_F}}\left(\sqrt{L_F+L_p\beta_{y}D_p+M_p^2\beta_y+\mu_F+L_pG_t}\right)+\frac{1}{\log 2}\right\}\\
    &+\sqrt{\frac{8\gamma_u}{\mu_F}}\left(1+\sqrt{\frac{\rho}{\underline{L}\sigma^t}}+\frac{2}{\underline{L}\sigma^t\beta_y}\right)H_t\Bigg\}\log 8\sqrt{\frac{\gamma_u L_{\bar{\psi}_{t,1}} D_y^2}{2\mu_F}}\frac{(L_F+M_p)(1+\gamma_u) L_{\bar{\psi}_{t,1}}}{\omega_{T_\eta} \mu_F}+{\color{black}3}(T_\eta+1)\Bigg).
\end{align*}
Finally, using $K_{y,t}\le\sum_{k=0}^{T_\eta}2n_{LS,x,t}K_{y,t,k}$ and plugging the notations in \eqref{eq:temp-notation} with explicit dependence on~$t$ complete the proof.
\end{proof}
% \begin{remark}
%     Although Algorithm~\ref{al:iNAPG} uses line search to estimate the Lipschitz constant for each subproblem in~\eqref{eq:cond-x-u}, this does not completely remove the need for Lipschitz information. In particular, to guarantee that the approximate gradient error is bounded by $\omega$, Lemma~\ref{lemma1} requires solving $\max_\vy Q(\vx,\vy,\vu)$ to an $\frac{\omega \mu_F}{L_F+M_p}$-stationary point, which in turn requires the knowledge of $L_F$ and $M_p$ to determine the target accuracy. Therefore, even with line search in Algorithm~\ref{al:iNAPG}, we still cannot claim that no Lipschitz information is needed.
% \end{remark}

Combining the outer iteration complexity result in Theorem~\ref{thm:outer-iter} and the complexity result for each outer iteration in Theorem~\ref{thm:each iteration complexity}, we are now ready to derive the total complexity of Algorithm~\ref{al:IAL} to find an $\varepsilon$-KKT point of \eqref{Basic Formulation}.
\begin{theorem}[Total complexity]\label{thm:total complexity}Suppose Assumptions~\ref{as1}--\ref{as2} hold. Let $\varepsilon>0$ be given and $N$ be defined in \eqref{eq:def-N}. Then to produce an $\varepsilon$-KKT point of \eqref{Basic Formulation}, Algorithm~\ref{al:IAL} needs $K_{x}$ calls to $\nabla_\vx F$, $J_\vc$, and $\prox_{\alpha_x g}$, and $K_{y}$ calls to $\nabla_\vy F$, $J_\vp$, and $\prox_{\alpha_y h}$, where
\begin{align}
K_{x} \le & ~2N \bar{n}_{LS,x} \left(\bar{T}+1\right) = \cO\left((\ln\tfrac{1}{\varepsilon})^2\frac{1}{\varepsilon}\right) = \tilde\cO(\varepsilon^{-1}), \label{eq:total-x-complexity}\\
        K_{y}\leq&~4N \bar{n}_{LS,x} \bar{n}_{LS,y} \Bigg(\Bigg\{ (\bar{T}+1)\left\{\sqrt{\frac{8\gamma_u}{\mu_F}}\left(\sqrt{L_F+L_p\beta_{y}D_p+M_p^2\beta_y+\mu_F+L_pG_0}\right)+\frac{1}{\log 2}\right\} \label{eq:total-y-complexity}\\
    &+\sqrt{\frac{8\gamma_u}{\mu_F}}\left(1+\sqrt{\frac{\rho_0}{\underline{L}}}+\frac{2}{\underline{L}\beta_y}\right){\color{black}\left(\frac{4\bar{E}_{0}\gamma_uL_{\psi,N}}{\rho_0 \sigma^{-N}}\right)^{\tfrac{1}{4}}\left(2\sqrt{\bar{T}}\right)}\Bigg\}\notag\\
    &\cdot\log 8\sqrt{\frac{\gamma_u \bar{L}_{\bar{\psi}_{1}} D_y^2}{2\mu_F}}\frac{(L_F+M_p)(1+\gamma_u) \bar{L}_{\bar{\psi}_{1}}}{\underline{\omega} \mu_F}+{\color{black}3}(\bar{T}+1)\Bigg) + N \bar{K}_{y}' \notag\\
   = & ~ \cO\left(\frac{1}{\varepsilon\sqrt{\mu_F}}\left(\ln\frac{1}{\varepsilon}\right)^2 \left(\ln\frac{1}{\mu_F} + \ln\frac{1}{\varepsilon}\right)^2 \ln\frac{1}{\mu_F}  \right) = \tilde\cO\left(\frac{1}{\varepsilon\sqrt{\mu_F}}\right) .
\end{align}
Here, the involved constants are defined as follows 
\begin{align}
&\bar{n}_{LS,x} = \left\lceil \log_{\gamma_u}\frac{L_{\psi,0}}{\underline{L}}\right\rceil+1,\quad \bar{n}_{LS,y} = \left\lceil \log_{\gamma_u}\frac{\bar{L}_{\bar{\psi}_{1}}}{\underline{L}}\right\rceil+1, \\
&\bar{K}_{y}' =  2 \left(\left\lceil \log_{\gamma_u}\frac{L_{\bar{\psi}}}{\underline{L}}\right\rceil+1\right)\left(\left\lceil \left(\sqrt{\frac{8\gamma_u L_{\bar{\psi}}}{\mu_F}} + \frac{1}{\log 2}\right) \log 8\sqrt{\frac{\gamma_u L_{\bar{\psi}} D_y^2}{2\mu_F}}\frac{(1+\gamma_u) L_{\bar{\psi}}}{\eta_N \mu_F^2}\right\rceil + {\color{black}2}\right),\label{eq:def-K-prime}\\
&L_{\bar{\psi}}=L_F+L_p(\beta_{y}D_p+D_u)+M_p^2\beta_y+\mu_F = \cO(1), \label{eq:def-L_bar-psi}\\
%{L}_{\psi,N}=\left(L_c\beta_{x,0}D_c+M_c^2\beta_{x,0}\right)\sigma^N+L_{\Psi}+L_cD_z ,
&\bar{L}_{\bar{\psi}_{1}}=L_F+L_p(\beta_{y}D_p+\bar{B}_1)+M_p^2\beta_y+\mu_F,\ 
\bar{B}_1=\sqrt{\frac{\bar{E}_{0}\gamma_uL_{\psi,N}}{\rho_0 \sigma^{-N}}}\left(1+\sqrt{\frac{\rho_0}{\underline{L}}}+\frac{2}{\underline{L}\beta_y}\right)+G_0,\\
&{L}_{\psi,0} \text{ and } {L}_{\psi,N} \text{ given in \eqref{eq:def-L-psi-Lbar-psi}},\ G_0=\cO(1) \text{ given in \eqref{eq:G-t}},\\
%G_0=\frac{2}{\underline{L}}\left(M_F + M_c(\beta_{x,0} D_c + D_z) +D_p + \frac{D_u}{\beta_y}\right)+\sqrt{\frac{1}{\underline{L}}}\left(\sqrt{2\nabla_g}+\sqrt{\rho_0}D_x+2\sqrt{\rho_0}D_u\right)+D_u,\\
% \bar{H}=\sqrt{\frac{\bar{E}_{0}\gamma_uL_{\psi,N}}{\rho_0 \sigma^N}}\Bigg\{2\left(\ln\left\{\sqrt{\frac{2\gamma_uL_{\psi,N}}{\rho_0\sigma^N}}\ln\left(\sqrt{\frac{2\gamma_uL_{\psi,N}}{\rho_0\sigma^N}}\right)+1\right\}+1\right)+\left(1+\sqrt{\frac{\rho_0}{2\gamma_uL_{\psi,0}}}\right)\Bigg\},\\
&\underline{\omega}= \min\left\{\frac{\eta_0\tau^{-N}}{8},\, \frac{\rho_0 \eta_0\sigma^{-N}\tau^{-N}}{(\bar{T}+1)^2\cdot\max\{1, \bar{A}\}}\right\},\\
&    \bar{T} = \left\lceil \left(\sqrt{\frac{8\gamma_u {L}_{\psi,N}}{\rho_0 \sigma^{-N}}} + \frac{1}{\log 2}\right) \log 8\sqrt{\frac{\gamma_u {L}_{\psi,N} \bar{E}_{0}}{\rho_0 \sigma^{-N}}}\frac{(1+\gamma_u) {L}_{\psi,N}}{\eta_0 \tau^{-N}}\right\rceil + 1,\\
&        \bar{A}=\left(1 + \frac{2\rho_0}{\underline{L}} + 2\sqrt{\frac{2\rho_0}{\underline{L}}}\right)^2 \left(\frac{128 \bar{E}_{0}}{\rho_0 \sigma^{-N}} \frac{(1+\gamma_u)^2 L_{\psi,N}^2}{\eta_0^2\tau^{-2N}} + \frac{2 + \sqrt{\frac{\underline{L}}{8\rho_0 \sigma^{-N}}}}{2\rho_0 \sigma^{-N} + 2\sqrt{2\rho_0 \sigma^{-N} \underline{L}}}\right),\\
& \bar{E}_{0}=D_x^{2}+2D_u^{2}
 + 30\bar{R} + \frac{8\rho_0^2 \eta_0^2}{\underline{L}} = \cO(1), \\
&\bar{R} = \rho_0\eta_0 \max\Bigg\{2D_x^{2}+8D_u^{2} + \sqrt{\frac{64  \rho_0^2 \eta_0^2 }{\underline{L}}},\; 240\rho_0\eta_0, \; 120 \eta_0\gamma_u L_{\psi,0}, \nonumber\\
&\qquad\qquad\qquad \qquad \frac{\sqrt{2\gamma_u L_{\psi,0}}(D_x^{2}+4D_u^{2})}{\sqrt{\rho_0 }} + 16\eta_0 \gamma_u L_{\psi,0} + \sqrt{\frac{32\eta_0^2 \rho_0 \gamma_u L_{\psi,0}}{\underline{L}}}\Bigg\} = \cO(1).
\end{align}

\end{theorem}
\begin{proof}
With the parameters $\beta_{x,t} = \beta_{x,0} \sigma^t$, $\rho_t = \rho_0 \sigma^{-t}$, and $\eta_t = \eta_0 \tau^{-t}$ given in Algorithm~\ref{al:IAL}, we have that for any $0\le t\le N$, the constants defined in \eqref{eq:n-LS-x-y-t}--\eqref{eq:def-R-t} satisfy
\begin{align*}
& \frac{{L}_{\psi,t}}{\sigma^t} \le {L}_{\psi,0},   {L}_{\psi,t}\leq {L}_{\psi,N},\  L_{\bar{\psi}_{t,1}}\leq\bar{L}_{\bar{\psi}_{1}},\ A_{T_{\eta_t}+1} \le \bar{A}, \ T_{{\eta_t}} \le \bar{T}, \\
& G_t\leq G_0,\ B_{t,1}\leq \bar{B}_1,\   \omega_{T_{\eta_t}}\geq\overline{\omega},\ \rho_t\eta_t R_t \le \bar{R}, \ {E}_{0,t}\leq \bar{E}_{0}.
\end{align*}
Hence, $n_{LS,x,t} \le \bar{n}_{LS,x}$, and it follows from Theorems~\ref{thm:outer-iter} and \ref{thm:each iteration complexity} that $K_x \le \sum_{t=0}^{\color{black}N-1} K_{x,t} \le 2N \bar{n}_{LS,x} \left(\bar{T}+1\right)$. Also, from the definition of $N$ in \eqref{eq:def-N} and $\tau > \sigma$, it holds \begin{equation}\label{eq:order-constats}
\begin{aligned}
&\textstyle N = \Theta(\ln\frac{1}{\varepsilon}),\ \sigma^N = \Theta(\frac{1}{\varepsilon}),\ L_{\psi,0} = \Theta(\frac{1}{\mu_F}),\ L_{\psi,N} = \Theta(\frac{1}{\mu_F} + \frac{1}{\varepsilon}),\ \bar{n}_{LS,x} = \Theta(\ln\frac{1}{\mu_F}), \\
&\textstyle \bar{B}_1 = \Theta\left(\sqrt{\frac{1}{\varepsilon}(\frac{1}{\mu_F} + \frac{1}{\varepsilon})}\right), \
 \bar{L}_{\bar{\psi}_{1}} = \Theta\left(\sqrt{\frac{1}{\varepsilon}(\frac{1}{\mu_F} + \frac{1}{\varepsilon})}\right), \ \bar{n}_{LS,y}= \Theta\left(\ln\big(\frac{1}{\varepsilon}(\frac{1}{\mu_F} + \frac{1}{\varepsilon})\big)\right), \\
&\textstyle \bar{A} = \Theta\left(\frac{1}{\varepsilon}(\frac{1}{\mu_F^2}+\frac{1}{\varepsilon^2})(\frac{1}{\varepsilon})^{{\color{black}2}\frac{\ln \tau}{\ln \sigma}}\right),\ \bar{T} = \Theta(\frac{1}{\varepsilon}\ln\frac{1}{\varepsilon}), \frac{1}{\underline{\omega}} = \Theta\left(\bar{A}\bar{T}^2 (\frac{1}{\varepsilon})^{\frac{\ln \tau}{\ln \sigma}+1} \right), \ \bar{K}_{y}' = \Theta\left(\sqrt{\frac{1}{\mu_F}}\ln \frac{1}{\mu_F\varepsilon}\right).  \end{aligned} 
\end{equation}  
Therefore, $K_x = \cO\left((\ln\tfrac{1}{\varepsilon})^2\frac{1}{\varepsilon}\right)$.

Moreover, let $K_{y,t}'$ be the number of $\vy$-gradients needed to obtain $\vy^{t+1}$ in Algorithm~\ref{al:IAL}. Then the total number of $\vy$-gradient evaluation $K_y \le \sum_{t=0}^N (K_{y,t} + K_{y,t}')$. Because $\|\vu^{t+1}\|\le D_u$ by Lemma~\ref{appendixboundofu}, it follows from Lemma~\ref{lm:Lipconstant} that $L_{\bar{\psi}}$ given in \eqref{eq:def-L_bar-psi} is the Lipschitz constant of $\nabla\bar{\psi}_t$. By Theorem~\ref{thm:iter-iAPG}, we have $ K_{y,t}' \le  \bar{K}_{y}'$ for all $0\le t\le N$ by the definition in \eqref{eq:def-K-prime}. 
Therefore, the total number, $K_y$, of $\vy$-gradient evaluation satisfies the bound in \eqref{eq:total-y-complexity}. 

Finally, using the bounds in \eqref{eq:order-constats}, we obtain $$K_y = \cO\left(\ln\frac{1}{\varepsilon}\ln\frac{1}{\mu_F}\ln\big(\frac{1}{\varepsilon}(\frac{1}{\mu_F} + \frac{1}{\varepsilon})\big)\left(\frac{1}{\varepsilon\sqrt{\mu_F}}\ln\frac{1}{\varepsilon}\right)\cdot\big(\ln\frac{1}{\mu_F} + \ln\frac{1}{\varepsilon}\big)\right)$$ and thus  complete the proof.
\end{proof}

The following theorem shows that the $\varepsilon$-KKT point output by our algorithm also implies a primal  $\cO(\varepsilon)$-optimal solution.

\begin{theorem}\label{thm:eps-optimal}
Under Assumptions~\ref{as1}--\ref{as2}, let $(\tilde\vx,\tilde\vy,\tilde\vz,\tilde\vu)$ be the output of Algorithm~\ref{al:IAL}. Then $\tilde\vx$ is a primal $\varepsilon_P$-optimal solution of problem~\eqref{Basic Formulation}, with
    \begin{equation*}
        \varepsilon_P=\max\big\{D_u' + 2+D_x(1+L_F)  + (\beta_y D_p + D_u + D_u')M_p, \ D'_z\big\}\varepsilon,
    \end{equation*}
    where $D_u'$ and $D'_z$ are defined in Lemmas~\ref{appendixboundofu} and \ref{appendixboundofz}. %~\eqref{eq:def-u'} and \eqref{eq:Dz'}.
\end{theorem}
\begin{proof}
From Theorem~\ref{thm:outer-iter}, $(\tilde\vx, \tilde\vy)$ is an $\varepsilon$-KKT point of \eqref{Basic Formulation} with corresponding multipliers $(\tilde\vz, \tilde{\vu})$.  Thus %we have \(\|[\vc(\tilde\vx)]_+\|\le \varepsilon\). 
\begin{align}
&  \dist\left(\vzero,\nabla_{\vx}F(\tilde\vx,\tilde\vy)+\partial g(\tilde\vx)+\mathbf{J}^\top_\mathbf{c}(\tilde\vx)\tilde\vz\right)\leq \varepsilon,\ \|[\mathbf{c}(\tilde\vx)]_+\|\leq \varepsilon,\ \textstyle \sum_{i=1}^q|\mathbf{c}_i(\tilde\vx)\tilde\vz_i|\leq \varepsilon, \label{eq:kkt-xtilde}\\
&  \dist\left(\vzero,\nabla_{ \vy}F(\tilde\vx,\tilde\vy)-\partial h(\tilde\vy)-\mathbf{J}^\top_\vp(\tilde\vy)\tilde\vu\right)\leq \varepsilon,\ 
  \|[\mathbf{p}(\tilde\vy)]_+\|\leq \varepsilon,\ 
\textstyle  \sum_{i=1}^r|\mathbf{p}_i(\tilde\vy)\tilde\vu_i|\leq \varepsilon. \label{eq:kkt-ytilde}
\end{align}
%
%For convenience, we first recall the definitions in \eqref{eq:defofQ} and \eqref{newproblem}, and then introduce the following notation:
Define
\begin{align}\label{eq:thm311eq1}
    %Q(\vx,\vy,\vu)=F(\mathbf{x}, \mathbf{y}) - h(\mathbf{y}) - \frac{1}{2 \beta_y} \left( \|[\beta_y \mathbf{p}(\mathbf{y}) + \mathbf{u}]_+\|^2 - \|\mathbf{u}\|^2 \right),\ 
    \bar{\vy}=\argmax_\mathbf{y}Q(\tilde\vx,\vy,\tilde\vu),\ \quad
    %\zeta(\mathbf{x})=\max_{\mathbf{y}}\left\{F(\mathbf{x},\mathbf{y})-h(\mathbf{y}), ~\mathrm{s.t.}~  \mathbf{p}(\mathbf{y}) \leq \vzero \right\},\ 
    \hat{\vy}=\argmax_\vy\left\{F(\tilde\vx,\mathbf{y})-h(\mathbf{y}), ~\mathrm{s.t.}~  \mathbf{p}(\mathbf{y}) \leq \vzero \right\}.%\label{eq:thm311eq1}
\end{align}
Then from \eqref{eq:boundofyy}, it holds
\begin{equation}\label{eq:diff-y-ytilde}
\|\tilde\vy - \bar\vy\| \le \vareps.     
\end{equation}
In addition, let $(\vx^*,\vy^*)$ be an optimal solution to \eqref{Basic Formulation} with a corresponding multiplier $(\vz^*,\vu^*)$. %be the corresponding optimal dual variables. %By Danskin's theorem, 
Then there exist $\vxi_\vx^*\in\partial g(\vx^*)$ and $\vxi_\vy^*\in\partial h(\vy^*)$ such that the following KKT conditions hold at $(\vx^*,\vy^*)$:
\begin{align}
    &\nabla_{\vx}F(\mathbf{x}^*,\mathbf{y}^*)+\vxi_\vx^*+\mathbf{J}^\top_\mathbf{c}(\mathbf{x}^*)\mathbf{z}^*=\vzero, \quad \vz^*\geq \vzero,\label{eq:KKTeq001}\\
    &\nabla_{\vy}F(\mathbf{x}^*,\mathbf{y}^*)-\vxi_\vy^*-\mathbf{J}^\top_\mathbf{p}(\mathbf{y}^*)\mathbf{u}^*=\vzero,\quad \vu^*\geq \vzero,\label{eq:KKTeq002}\\
    &\vc_i(\vx^*)\vz_i^*=0,\ i=1,\cdots,q,\quad \vp_i(\vy^*)\vu_i^*=0,\ i=1,\cdots,r,\notag\\
    %&\vz^*\geq \vzero,\ \vu^*\geq \vzero,\notag\\
    &\vc(\vx^*)\leq \vzero,\ \vp(\vy^*)\leq \vzero.\notag
\end{align}

By~\eqref{eq:KKTeq002}, we have
\begin{align}
    0&=\ip{\vy^*-\bar{\vy}}{\nabla_{\vy}F(\mathbf{x}^*,\mathbf{y}^*)-\vxi_\vy^*-\mathbf{J}^\top_\mathbf{p}(\mathbf{y}^*)\mathbf{u}^*}\notag\\
    &\leq F(\vx^*,\vy^*)-F(\vx^*,\bar{\vy})-\frac{\mu_F}{2}\|\bar{\vy}-\vy^*\|^2-h(\vy^*)+h(\bar{\vy})+\sum_{i=1}^r\vu_i^{*}(\vp_i(\bar{\vy})-\vp_i(\vy^*)).\label{eq:KKTeq3}
\end{align} 
Also, by %the definition of the $\varepsilon$-KKT point and the definition of $\bar\vy$ in \eqref{eq:thm311eq1}, 
\eqref{eq:kkt-xtilde} and \eqref{eq:diff-y-ytilde}, there exists $\vxi_\vx\in\partial g(\tilde\vx)$ such that $\|\nabla_{\vx}F(\tilde\vx,\bar{\vy})+\vxi_\vx+\mathbf{J}^\top_\mathbf{c}(\tilde\vx)\tilde\vz\| \le \|\nabla_{\vx}F(\tilde\vx,\tilde{\vy})+\vxi_\vx+\mathbf{J}^\top_\mathbf{c}(\tilde\vx)\tilde\vz\| + L_F\|\bar{\vy} - \tilde\vy\| \le \varepsilon + L_F\varepsilon$. Hence,
\begin{align}
    -D_x(\varepsilon + L_F\varepsilon) &\leq\ip{\vx^*-\tilde\vx}{\nabla_{\vx}F(\tilde\vx,\bar{\vy})+\vxi_\vx+\mathbf{J}^\top_\mathbf{c}(\tilde\vx)\tilde\vz}\notag\\
    &\leq F(\vx^*,\bar{\vy})-F(\tilde\vx,\bar{\vy})+g(\vx^*)-g(\tilde\vx)+\sum_{i=1}^q \tilde\vz_i(\vc_i(\vx^*)-\vc_i(\tilde\vx)).\label{eq:KKTeq4}
\end{align}
Denote $\bar{\vu}\coloneqq[\beta_y\vp(\bar{\vy})+\tilde\vu]_+$. The optimality conditions for the two problems in~\eqref{eq:thm311eq1} imply that there exist $\bar\vxi_\vy\in\partial h(\bar{\vy})$,  $\hat\vxi_\vy\in\partial h(\hat{\vy})$, and $\hat{\vu}\ge\vzero$ such that
\begin{align}\label{eq:KKTeq003}
        \vzero=\nabla_\vy F(\tilde\vx,\bar{\vy})-\bar\vxi_\vy-\vJ_\vp(\bar{\vy})^\top \bar{\vu},\quad 
        \vzero=\nabla_\vy F(\tilde\vx,\hat{\vy})-\hat\vxi_\vy-\vJ_\vp(\hat{\vy})^\top \hat{\vu}. %\label{eq:KKTeq004}
    \end{align}
%where $\hat{\vu}$ is the corresponding dual multiplier associated with $\hat{\vy}$ for problem~\eqref{eq:thm311eq1}. 
By the first equation in \eqref{eq:KKTeq003}, we have
    \begin{align}
        0=&\ip{\bar{\vy}-\hat{\vy}}{\nabla_\vy F(\tilde\vx,\bar{\vy})-\bar\vxi_\vy-\vJ_\vp(\bar{\vy})^\top\bar {\vu}}\notag\\
        \leq&F(\tilde\vx,\bar{\vy})-F(\tilde\vx,\hat{\vy})-\frac{\mu_F}{2}\|\hat{\vy}-\bar{\vy}\|^2-h(\bar{\vy})+h(\hat{\vy})+\sum_{i=1}^r {\bar{\vu}}_i(-\vp_i(\bar{\vy})+\vp_i(\hat{\vy})).\label{eq:KKTeq2}
    \end{align}
Adding \eqref{eq:KKTeq3}, \eqref{eq:KKTeq4} and \eqref{eq:KKTeq2} and using the definition of $\zeta$ in \eqref{newproblem} gives
\begin{align*}
  (\zeta+g)(\tilde\vx) - (\zeta+g)(\vx^*) = & ~F(\tilde\vx,\hat{\vy})-h(\hat{\vy})+g(\tilde\vx)-(F(\vx^*,\vy^*)-h(\vy^*)+g(\vx^*))\notag\\
    \leq &~\sum_{i=1}^r\vu_i^{*}(\vp_i(\bar{\vy})-\vp_i(\vy^*))+\sum_{i=1}^q \tilde\vz_i(\vc_i(\vx^*)-\vc_i(\tilde\vx))+D_x(1+L_F)\varepsilon \notag\\
    &~ -\frac{\mu_F}{2}\|\bar{\vy}-\vy^*\|^2 -\frac{\mu_F}{2}\|\hat{\vy}-\bar{\vy}\|^2+\sum_{i=1}^r{\bar{\vu}}_i(-\vp_i(\bar{\vy})+\vp_i(\hat{\vy})). %\\
    %\leq &{\color{black}(D'_u+1+D_x+3\beta_yD_p+D_u)\varepsilon {should be deleted}}
\end{align*}
Noticing $\vu_i^* \vp_i(\vy^*) = 0,\forall\, i$, $\tilde\vz_i \vc_i(\vx^*) \le 0, \forall\, i$, and $\bar\vu_i \vp_i(\hat\vy) \le 0, \forall\, i$, we obtain from the inequality above that
\begin{equation}\label{eq:zeta-g-x-xstar}
(\zeta+g)(\tilde\vx) - (\zeta+g)(\vx^*) \le \sum_{i=1}^r\vu_i^{*}\vp_i(\bar{\vy}) - \sum_{i=1}^q \tilde\vz_i \vc_i(\tilde\vx) +D_x(1+L_F)\varepsilon - \sum_{i=1}^r{\bar{\vu}}_i \vp_i(\bar{\vy}). 
\end{equation}

From Lemma~\ref{appendixboundofu}, \eqref{eq:kkt-ytilde}, and \eqref{eq:diff-y-ytilde}, it follows that 
\begin{equation}\label{eq:ineq-star-u-p-y}  
\sum_{i=1}^r\vu_i^{*}\vp_i(\bar{\vy}) \le \|\vu^*\| \cdot \|[\vp(\bar\vy)]_+\| \le D_u' \|[\vp(\bar\vy)]_+\| \le D_u' \|[\vp(\tilde\vy)]_+\| + D_u' M_p\varepsilon \le D_u'(1+M_p)\varepsilon.
\end{equation}
Also, let $I_{-+} =\{i: \vp_i(\bar\vy) < 0, \bar\vu_i >0 \}$, then we have
\begin{align}
& ~ - \sum_{i=1}^r{\bar{\vu}}_i \vp_i(\bar{\vy}) \le - \sum_{i\in I_{-+}} {\bar{\vu}}_i \vp_i(\bar{\vy}) =  - \sum_{i\in I_{-+}} \big(\beta_y \vp_i(\bar{\vy}) + \tilde\vu_i\big) \vp_i(\bar{\vy})  \le - \sum_{i\in I_{-+}} \tilde\vu_i \vp_i(\bar{\vy}) \notag\\
= & ~ - \sum_{i\in I_{-+}} \tilde\vu_i \vp_i(\tilde{\vy}) + \sum_{i\in I_{-+}} \tilde\vu_i \big(\vp_i(\tilde{\vy})-\vp_i(\bar{\vy})\big) \le \varepsilon + (\beta_y D_p + D_u)M_p \varepsilon, \label{eq:ineq-bar-u-p-y}
\end{align}
where the last inequality follows from \eqref{eq:kkt-ytilde}, \eqref{eq:diff-y-ytilde}, Lemma~\ref{appendixboundofu}, and the $M_p$-Lipschitz continuity of $\vp$. Now plugging \eqref{eq:ineq-star-u-p-y} and \eqref{eq:ineq-bar-u-p-y} into \eqref{eq:zeta-g-x-xstar} and using \eqref{eq:kkt-xtilde}, we obtain 
\begin{equation}\label{eq:zeta-g-x-xstar2}
(\zeta+g)(\tilde\vx) - (\zeta+g)(\vx^*) \le D_u'(1+M_p)\varepsilon + (1+D_x(1+L_F))\varepsilon + \varepsilon + (\beta_y D_p + D_u)M_p \varepsilon.
\end{equation}
    
%     By~\eqref{eq:inequality for bound U and p(y)} and~\eqref{eq:B_1}, we have
%     \begin{align}
%         \sum_i|\mathbf{p}_i(\bar{\vy})\bar{\vu}_i|&= \sum_{i\in I_1}|\mathbf{p}_i(\bar{\vy})\bar{\vu}_i|+\sum_{i\in I_2}|\mathbf{p}_i(\bar{\vy})\bar{\vu}_i|\notag\\
%         &\leq (\sum_{i\in I_1}\mathbf{p}_i^2(\bar{\vy})\sum_{i\in I_1\cup I_2}\bar{\vu}_i^2)^\frac{1}{2}+(\sum_{i\in I_1\cup I_2}\mathbf{p}_i^2(\bar{\vy})\sum_{i\in I_2}\bar{\vu}_i^2)^\frac{1}{2}\notag\\
%         &\leq \varepsilon(\beta_yD_p+D_u)+ D_p(2\beta_y\varepsilon\notag)\\
%         &=(3\beta_yD_p+D_u)\varepsilon\label{eq:KKTeq005}
%     \end{align}

% The last inequality follows from $\vu^*\geq \vzero$, $\|[\vp(\bar{\vy})]_+\|\leq\varepsilon$~\eqref{eq:condition3}, $\vu^*\vp(\vy^*)^\top=0$, $\vz\geq \vzero$, $\vc(\vx^*)\leq \vzero$, $|\vz \vc(\vx)^\top|\leq \varepsilon$~\eqref{eq:condition52}
% , $\sum_i|\mathbf{p}_i(\bar{\vy})\bar{\vu}_i|\leq \varepsilon (3\beta_yD_p+D_u)$~\eqref{eq:KKTeq005}, $\bar{\vu}\geq \vzero$ and $\vp(\hat{\vy})\leq \vzero$.

On the other hand, by~\eqref{eq:KKTeq001}, we have
\begin{align*}
    0&=\ip{\tilde\vx-\vx^*}{\nabla_{\vx}F(\mathbf{x}^*,\mathbf{y}^*)+\vxi_\vx^*+\mathbf{J}^\top_\mathbf{c}(\mathbf{x}^*)\mathbf{z}^*}\notag\\
    &\leq F(\tilde\vx,\vy^*)-F(\vx^*,\vy^*)+g(\tilde\vx)-g(\vx^*)+\sum_{i=1}^q\vz^*_i(\vc_i(\tilde\vx)-\vc_i(\vx^*)).
\end{align*}
Also by \eqref{eq:KKTeq003}, we have
\begin{align*}
    %0&{\color{red}=\ip{\hat{\vy}-\vy^*}{\nabla_\vy F(\tilde\vx,\hat{\vy})-\tilde{\nabla}h(\hat{\vy})-\vJ_\vp(\hat{\vy})^\top \hat{\vu}}}\\
    0&=\ip{\hat{\vy}-\vy^*}{\nabla_\vy F(\tilde\vx,\hat{\vy})-{\color{black}\hat\vxi_\vy}-\vJ_\vp(\hat{\vy})^\top \hat{\vu}}\\
    &\leq F(\tilde\vx,\hat{\vy})-F(\tilde\vx,\vy^*)-\frac{\mu_F}{2}\|\vy^*-\hat{\vy}\|^2-h(\hat{\vy})+h(\vy^*)+\sum_{i=1}^r\hat{\vu}_i(-\vp_i(\hat{\vy})+\vp_i(\vy^*)).
\end{align*}
Adding the above two inequalities gives
\begin{align*}
  (\zeta+g)(\vx^*) - (\zeta+g)(\tilde\vx) =   & ~F(\vx^*,\vy^*)-h(\vy^*)+g(\vx^*)-(F(\tilde\vx,\hat{\vy})-h(\hat{\vy})+g(\tilde\vx))\notag\\
    \leq &\sum_{i=1}^q\vz_i^*(\vc_i(\tilde\vx)-\vc_i(\vx^*))-\frac{\mu_F}{2}\|\vy^*-\hat{\vy}\|^2+\sum_{i=1}^r\hat{\vu}_i(-\vp_i(\hat{\vy})+\vp_i(\vy^*))\\
    \le & \sum_{i=1}^q\vz_i^*\vc_i(\tilde\vx) \le \|\vz^*\|\cdot\|[\vc(\tilde\vx)]_+\|
    \leq D'_z\varepsilon,
\end{align*}
where the second inequality follows from $\vc_i(\vx^*)\vz_i^*=0, \forall\, i$, $\vp_i(\hat{\vy})\hat{\vu}_i=0, \forall\, i$, $\hat{\vu}\geq\vzero$ and $\vp(\vy^*)\leq\vzero$, and the last inequality holds by Lemma~\ref{appendixboundofz} and \eqref{eq:kkt-xtilde}. The above inequality together with \eqref{eq:kkt-xtilde} and \eqref{eq:zeta-g-x-xstar2} indicates the desired result.
\end{proof}

\subsection{Convex Concave Case}
%When Assumption~\ref{as2} is violated, $\mathbf{F(\mathbf{x},\cdot)}$ Remains concave but not strongly concave. 
In this subsection, we consider the case where $F$ may not be strongly concave but merely concave about $\vy$. To apply the results established in the previous subsection,
%To address this, 
we perturb the $\vy$-problem by adding a (tolerance-dependent) small quadratic %regularization 
term, thereby ensuring strong concavity. %, and then apply the algorithm developed for the strongly concave case. We will show that this modification yields the $\varepsilon$-KKT point and establish its corresponding complexity bound.

More specifically, let $\vy^0\in\dom (h)$ and $\nu>0$. Then we consider
\begin{equation}\label{Perturb Formulation}
    \min_{\mathbf{x}\in\RR^n}\max_{\mathbf{y}\in\RR^m}F(\mathbf{x}, \mathbf{y}) -\frac{\nu}{2}\|\mathbf{y}-\mathbf{y}^0\|^2 - h(\mathbf{y})+g(\mathbf{x}), \quad \text{s.t.} \quad \mathbf{p}(\mathbf{y}) \leq \vzero, \quad \mathbf{c}(\mathbf{x}) \leq \vzero.
\end{equation}
By Definition~\ref{def1} and the boundedness of $\dom(h)$, it is easy to see that if $\nu \le \frac{\varepsilon}{2 D_y}$ and $(\tilde\vx,\tilde\vy)$ is an $\frac{\varepsilon}{2}$-KKT point of problem \eqref{Perturb Formulation}, then it must be an $\varepsilon$-KKT point of \eqref{Basic Formulation}. Hence, we can immediately obtain the following complexity result from Theorem~\ref{thm:total complexity}.
\begin{theorem}\label{thm:complexity-ccv}
Suppose Assumptions~\ref{as1}--\ref{slatercondition} hold. Let $\varepsilon>0$ be given. Choose $\nu =\frac{\varepsilon}{2D_y}$, where $D_y$ is defined in \eqref{sectionnotation}. Apply Algorithm~\ref{al:IAL} to produce an $\frac{\varepsilon}{2}$-KKT point $(\tilde\vx,\tilde\vy)$ of \eqref{Perturb Formulation}. Then $(\tilde\vx,\tilde\vy)$ is an $\varepsilon$-KKT point of \eqref{Basic Formulation}. In addition, the algorithm needs at most $\tilde\cO(\varepsilon^{-1})$ calls to $\nabla_\vx F, J_\vc$ and $\prox_{\alpha_x g}$ and $\tilde\cO(\varepsilon^{-\frac{3}{2}})$ calls to $\nabla_\vy F, J_\vp$ and $\prox_{\alpha_y h}$.   
\end{theorem}

%\YX{Give the complexity to produce an $\varepsilon$-optimal solution.}

For any $\vx$, define 
\begin{equation}\label{newproblem-nu}
    \zeta_\nu(\mathbf{x})=\max_{\mathbf{y}}\left\{F(\mathbf{x},\mathbf{y}) -\frac{\nu}{2}\|\mathbf{y}-\mathbf{y}^0\|^2-h(\mathbf{y}), ~\mathrm{s.t.}~  \mathbf{p}(\mathbf{y}) \leq \vzero \right\}.
\end{equation}
Then it holds $-\frac{\nu D_y^2}{2}\le \zeta_\nu(\mathbf{x}) - \zeta(\mathbf{x}) \le 0$ for all $\vx$, where $\zeta$ is defined in \eqref{newproblem}. Let $\vx^*$ be a minimizer of $\min_{\mathbf{x}}\big\{\zeta(\mathbf{x})+g(\mathbf{x}),  ~\mathrm{s.t.}~  \mathbf{c}(\mathbf{x}) \leq \vzero \big\}$ and $\vx_\nu^*$ be a minimizer of $\min_{\mathbf{x}}\big\{\zeta_\nu(\mathbf{x})+g(\mathbf{x}),  ~\mathrm{s.t.}~  \mathbf{c}(\mathbf{x}) \leq \vzero \big\}$. Then 
$$-\frac{\nu D_y^2}{2}\le(\zeta_\nu+g)(\vx_\nu^*) - (\zeta+g)(\vx_\nu^*)\le(\zeta_\nu+g)(\vx_\nu^*) - (\zeta+g)(\vx^*) \le (\zeta_\nu+g)(\vx^*) - (\zeta+g)(\vx^*) \le 0.$$
Therefore, if $\nu \le \frac{\varepsilon}{D_y^2}$ and $\tilde\vx$ is a primal $\frac{\varepsilon}{2}$-optimal solution of \eqref{Perturb Formulation}, then $\tilde\vx$ is a primal ${\varepsilon}$-optimal solution of \eqref{Basic Formulation}. Hence, from Theorems~\ref{thm:eps-optimal} and \ref{thm:complexity-ccv}, the complexity for Algorithm~\ref{al:IAL} to produce a primal ${\varepsilon}$-optimal solution of \eqref{Basic Formulation} is $\tilde\cO(\varepsilon^{-1})$ on $\vx$-part and $\tilde\cO(\varepsilon^{-\frac{3}{2}})$ on $\vy$-part.

\section{Concluding Remarks}
We have presented a first-order inexact proximal augmented Lagrangian method for solving a class of convex-(strongly-)concave minimax problems with functional constraints. By incorporating the functional constraints in the maximization variable into the objective, we obtain an equivalent reformulation, to which a proximal augmented Lagrangian method can be applied. Our method enjoys the near-optimal complexity of $\tilde\cO(\varepsilon^{-1})$ to produce either an %{\color{red}should be something like $\varepsilon$ ?}$\tilde\cO(\varepsilon^{-1})$
$\varepsilon$-KKT point or an $\varepsilon$-optimal solution, when the problem is convex-strongly-concave; the complexity will increase to $\tilde\cO(\varepsilon^{-\frac{3}{2}})$ for the dual gradient evaluations and remains at the near-optimal order for the primal evaluations, when the problem is convex-concave. The order $\tilde\cO(\varepsilon^{-1})$ is near-optimal even for functional constrained convex minimization problems. However, it is unknown whether in the convex-concave case, the number of dual gradient evaluation, i.e., $\tilde\cO(\varepsilon^{-\frac{3}{2}})$, is near-optimal or can be improved to $\tilde\cO(\varepsilon^{-1})$.

% In this paper, we proposed a first-order inexact proximal augmented Lagrangian method for solving a broad class of convex-(strongly-)concave minimax problems with functional constraints. By incorporating the functional constraints in the maximization variable into the objective, we obtain an equivalent reformulation to which a proximal AL method can be applied. The resulting value-function reformulation leads to inexact gradient evaluations, which are handled by an inexact accelerated proximal gradient method as an AL subproblem solver.

% We conducted a complexity analysis and derived oracle complexity bounds for finding an $\varepsilon$-KKT point of~\eqref{Basic Formulation}, measured by the number of gradient evaluations and proximal operator evaluations. Specifically, in the convex-strongly-concave case, the proposed algorithm requires $\tilde\cO(\varepsilon^{-1})$ first-order oracle calls. In the convex-concave case, the proposed algorithm requires $\tilde\cO(\varepsilon^{-1})$ calls to the $\vx$-part oracles and $\tilde\cO(\varepsilon^{-3/2})$ calls to the $\vy$-part oracles. These bounds provide complexity guarantees for deterministic minimax problems with explicit functional constraints beyond settings with easy-to-project constraints. We further showed that, in both cases, the point returned by the algorithm is also a primal $\varepsilon_P$-optimal point.

\appendix
\section{Inexact Nesterov’s Accelerated Proximal Gradient Method}
\xu{In \cite{AcceleratedMethods-Xu2026}}, an inexact accelerated proximal gradient method is given for %Consider 
the following problem 
\begin{equation}\label{eq:3}
\vv^*=\argmin_\mathbf{v}P(\mathbf{v}) \coloneqq \psi(\mathbf{v})+\phi(\mathbf{v}),
\end{equation}
where
$\psi(\cdot)$ is convex and has an $L_\psi$-Lipschitz continuous gradient, and
$\phi(\cdot)$ is closed and $\mu_\phi$-strongly convex and admits an easy proximal mapping, i.e.,
%\begin{equation*}
    $\mathbf{prox}_{\xu{\eta}\phi}(\va)\coloneqq\argmin_\vv \frac{1}{2\xu{\eta}}\|\vv-\va\|^2+\phi(\vv)$ 
%\end{equation*}
can be computed for any $\va$ \xu{and any $\eta>0$}. %In our setting, t
The exact gradient of $\psi$ may be not available; instead the approximate gradient of $\psi$ at any $\vv$, denoted as $\tilde{\nabla}\psi(\vv)$, can be \xu{accessed}. %assess. 
For any $\vv$, we denote
%\begin{equation*}
    $\ve(\vv)=\tilde{\nabla}\psi(\vv)-\nabla\psi(\vv)$ 
%\end{equation*}
as the error of the approximate gradient.
\begin{algorithm}[H]
\caption{%Inexact Accelerated Proximal Gradient Method with Line Search for \eqref{eq:3} 
$\hat\vv = \mathrm{iAPG}(\psi,\phi;\vv^0, L_{\min}, \mu_\phi, \varepsilon,\gamma_u, \gamma_d)$}
%$\hat{\vv} = \mathrm{iAPG}(\psi, g, \mu_\psi, L_\psi, \varepsilon, \{e_t\}, \{e_t'\}, \vv^0)$
\label{al:iNAPG}
\begin{algorithmic}[1]
\State \textbf{Input:} %Lipschitz constant 
 tolerance $\varepsilon>0$; $\vv^0 \in \dom(\phi)$; $\gamma_u>1$, $\gamma_d>1$, $L_{\min}>0$; $\mu_\phi > 0$. %nonnegative sequences $\{e_t\}$.
%\State \textbf{Output:} $\hat{\vx} = \mathrm{iAPG}(\psi, g, \mu_\psi, L_\psi, \varepsilon, \{e_t\}, \{e_t'\}, \vx^0)$
\State $\textbf{Set:}$ $A_0=0$, $\digamma_0(\vv)=\frac{1}{2}\|\vv-\vv^0\|^2$, $L_0\geq L_{\min}$, and $\bm{\pi}^0=\vv^0 $.
\For{$k = 0, 1, \ldots$}
    \State Let $L=L_k$. Set $\texttt{FLAG}=false$.
    \While {$\texttt{FLAG}==false$}
    \State Find $a>0$ such that $\frac{a^2}{A_k+a}=\frac{2(1+\mu_\phi A_k)}{L}$; let $\vb=\frac{A_k\vv^k+a\bm{\pi}^k}{A_k+a}$.
    \State Let $\omega = \min\left\{\frac{\varepsilon}{8},\, \frac{\mu_\phi \varepsilon}{(k+1)^2\cdot\max\{1, A_k + a\}}\right\}$; obtain $\tilde{\nabla} \psi(\vb)$ satisfying $\|\ve(\vb)\| \le \omega$.
    \State Set $\vd=\mathbf{prox}_{\phi/L}\left(\vb-\frac{1}{L}\tilde{\nabla}\psi(\vb)\right)$ and let $\xi=\tilde{\nabla}\psi(\vd)-\tilde{\nabla}\psi(\vb)-L(\vd-\vb)$.
    \State Obtain $\tilde{\nabla} \psi(\vd)$ satisfying $\|\ve(\vd)\| \le \omega$. %{\color{blue}/* Suppose $\|\ve(\vd)\|\leq e_t$ and $\|\ve(\vb)\|\leq e_t$ */}
    %\If{$\ip{\xi}{\vb-\vd}\geq\frac{1}{L}\|\xi\|^2-6\omega\|\vb-\vd\|-\frac{4\omega^2}{L}$}
    \State \textbf{if} $\ip{\xi}{\vb-\vd}\geq\frac{1}{L}\|\xi\|^2-6\omega\|\vb-\vd\|-\frac{4\omega^2}{L}$, \textbf{then} $\texttt{FLAG}=true$, \textbf{otherwise} $L=L\cdot \gamma_u$.
    %\Else
    %\State $L=L\cdot \gamma_u$.
    %\EndIf
    \EndWhile
    \State Let $\vb^k=\vb$, $\vv^{k+1}=\vd$, $M_k=L$, $L_{k+1}=\max\left\{L_{\min},L/\gamma_d \right\}$, $a_{k+1}=a$, $A_{k+1}=A_k+a$, and $\omega_k = \omega$.
    \State Set $\bm{\pi}^{k+1}=\argmin_\vv\digamma_{k+1}(\vv)\coloneqq \digamma_{k}(\vv)+a_{k+1}\left(\psi(\vv^{k+1})+\ip{\tilde{\nabla}\psi(\vv^{k+1})}{\vv-\vv^{k+1}}+\phi(\vv)\right)$.
    %\If{$\|\tilde{\nabla}\psi(\vd)-\tilde{\nabla}\psi(\vb)-L(\vd-\vb)\|\leq\frac{\varepsilon}{2}$}
    \State \textbf{if} $\|\tilde{\nabla}\psi(\vd)-\tilde{\nabla}\psi(\vb)-L(\vd-\vb)\|\leq\frac{\varepsilon}{2}$, \textbf{then} Return $\vd$ and Stop.
    %\EndIf
\EndFor
\end{algorithmic}
\end{algorithm}

The following results are established in \cite{AcceleratedMethods-Xu2026}. We will use them to establish our complexity bounds.

\begin{theorem}\label{thm:iter-iAPG}
Algorithm~\ref{al:iNAPG} can return an $\varepsilon$-stationary solution $\bar\vv$ of the objective $P(\cdot)$ of \eqref{eq:3} within ${\color{black}T_\eta+1}$ iterations, where $T_\varepsilon$ satisfies
\begin{equation}\label{eq:iter-complexity-k}
T_\varepsilon \le  \left\lceil \left(\sqrt{\frac{8\gamma_u L_\psi}{\mu_\phi}} + \frac{1}{\log 2}\right) \log 8\sqrt{\frac{\gamma_u L_\psi E_{0}}{\mu_\phi}}\frac{(1+\gamma_u) L_\psi}{\varepsilon}\right\rceil + 1,    
\end{equation}
with 
$E_0=\frac{1}{2}\|\vv^{*} - \vv^0\|^{2}
 + 30\mu_\phi\varepsilon R + \frac{8\mu_\phi^2 \varepsilon^2}{L_{\min}}$ and
\begin{align*}
R = & \max\Bigg\{2\|\vv^{*} - \vv^0\| + \sqrt{\frac{64  \mu_\phi^2 \varepsilon^2 }{L_{\min}}},\; 240\mu_\phi\varepsilon, \; 120 \varepsilon\gamma_u L_\psi, \nonumber\\
&\qquad\qquad\qquad \frac{\sqrt{2\gamma_u L_\psi}\|\vv^{*} - \vv^{0}\|}{\sqrt{\mu_\phi }} + 16\varepsilon \gamma_u L_\psi + \sqrt{\frac{32\varepsilon^2 \mu_\phi \gamma_u L_\psi}{L_{\min}}}\Bigg\}.
\end{align*}
In addition, the total number of inexact gradient evaluations is upper bounded by 
$$2\left(\left\lceil \left(\sqrt{\frac{8\gamma_u L_\psi}{\mu_\phi}} + \frac{1}{\log 2}\right) \log 8\sqrt{\frac{\gamma_u L_\psi E_{0}}{\mu_\phi}}\frac{(1+\gamma_u) L_\psi}{\varepsilon}\right\rceil + 1\right)\left(\left\lceil \log_{\gamma_u}\frac{L_\psi}{L_{\min}}\right\rceil+1\right).$$
{\color{black}In particular, if exact gradients of $\psi$ are used in the algorithm, the total number of gradient evaluations is upper bounded by
$$2\left(\left\lceil \left(\sqrt{\frac{8\gamma_u L_\psi}{\mu_\phi}} + \frac{1}{\log 2}\right) \log 8\sqrt{\frac{\gamma_u L_\psi \|\vv^*-\vv^0\|^2}{2\mu_\phi}}\frac{(1+\gamma_u) L_\psi}{\varepsilon}\right\rceil + 1\right)\left(\left\lceil \log_{\gamma_u}\frac{L_\psi}{L_{\min}}\right\rceil+1\right).$$}
Moreover, it holds $\|\vv^k - \vv^*\| \le \sqrt{\frac{2E_0}{\mu_\phi A_k}}$, $\|\bm{\pi}^k - \vv^*\| \le \sqrt{\frac{2E_0}{\mu_\phi A_k}}$, for all $k\ge1$, and
\begin{align}
&A_k \ge \max\left\{\frac{k^2}{2\gamma_u L_\psi},\ \frac{2}{\gamma_u L_\psi}\left(1+\sqrt{\frac{\mu_\phi}{2\gamma_u L_\psi}}\right)^{2(k-1)}\right\}, \, \forall\, k\ge1, \label{eq:low-bd-A} \\
&A_{k} \le \left(1 + \frac{2\mu_\phi}{L_{\min}} + 2\sqrt{\frac{2\mu_\phi}{L_{\min}}}\right)^2 \left(\frac{128 E_{0}}{\mu_\phi} \frac{(1+\gamma_u)^2 L_\psi^2}{\varepsilon^2} + \frac{2 + \sqrt{\frac{L_{\min}}{8\mu_\phi}}}{2\mu_\phi + 2\sqrt{2\mu_\phi L_{\min}}}\right), \forall\, k  \le T_\varepsilon+1. \label{eq:up-bd-cond-Ak}
\end{align}
\end{theorem}

\section{Bound on $\vu$ and $\vz$ iterates} In this section, we provide proofs of Lemmas~\ref{appendixboundofu}-\ref{appendixboundofz}.

\subsection{Proof of Lemma \ref{appendixboundofu}}
By the update rule in Algorithm~\ref{al:IAL} for $(\vx^{t+1}, \vu^{t+1})$ and \eqref{eq:thenablaofpsi}, we have 
\begin{equation*}
\|\nabla_{\mathbf{u}}\Psi(\mathbf{x}^{t+1},\mathbf{u}^{t+1})+\rho_t(\mathbf{u}^{t+1}-\mathbf{u}^{t})\| = \left\|\frac{1}{\beta_y}(\mathbf{u}^{t+1} -[\beta_y \vp(\bar{\mathbf{y}}^{t+1}) + \mathbf{u}^{t+1}]_+)+\rho_t(\mathbf{u}^{t+1}-\mathbf{u}^{t})\right\|\leq \eta_t,
\end{equation*}
% indicates that
% \begin{equation*} \label{it:outiteration}
%     \dist( \vzero, \partial_{(\vx,\vu)} \left[ \Psi(\mathbf{x}^{t+1},\mathbf{u}^{t+1}) + g(\mathbf{x}^{t+1}) 
%     + \| [\beta_{x,t} \mathbf{c}(\mathbf{x}^{t+1}) + \mathbf{z}^{t}]_+ \|^2 
%     + \frac{\rho_t}{2} (\mathbf{x}^{t+1} - \mathbf{x}^t)^2+\frac{\rho_t}{2} (\mathbf{u}^{t+1} - \mathbf{u}^t)^2
%      \right] ) \leq \eta_t.
% \end{equation*}
% By considering the $\vu$-component of the subdifferential and noting that its Euclidean norm is bounded by that of the full subgradient, we obtain
where we 
let $\bar{\mathbf{y}}^{t+1}\coloneqq\mathbf{y}(\mathbf{x}^{t+1}, \mathbf{u}^{t+1})$ by the definition in \eqref{eq:solve y}. Define $\bar{\vu}^{t+1}\coloneqq[\beta_y \vp(\bar{\mathbf{y}}^{t+1}) + \mathbf{u}^{t+1}]_+$. Then noticing $\vu^{t+1} = \frac{1}{1+\rho_t\beta_y} \left(\bar{\vu}^{t+1} + \rho_t \beta_y \vu^t + (\vu^{t+1} - \bar{\vu}^{t+1}) + \rho_t \beta_y(\mathbf{u}^{t+1}-\mathbf{u}^{t})\right)$, we have from the inequality above and the triangle inequality that %an inequality which will be used later:
\begin{equation}\label{in:u}
    \|\vu^{t+1}\|\leq\frac{1}{1+\rho_t\beta_y}\|\bar{\vu}^{t+1}\|+\frac{\rho_t\beta_y}{1+\rho_t\beta_y}\|\vu^t\|+\frac{\beta_y\eta_t}{1+\rho_t\beta_y}.
\end{equation}

%In addition, 
Also, by the definition of $\bar{\vy}^{t+1}$, there exists $\bm{\xi}\in \partial h(\bar{\mathbf{y}}^{t+1})$ such that $\nabla_\mathbf{y}F(\mathbf{x}^{t+1},\bar{\mathbf{y}}^{t+1})-\bm{\xi}-\mathbf{J}^\top_\vp(\bar{\mathbf{y}}^{t+1})\bar{\vu}^{t+1}=\vzero.$ Thus by the concavity of $F(\mathbf{x}^{t+1},\cdot)$ and the convexity of $h$ and each component of $\vp$, it holds 
\begin{align}
    0=&\langle \bar{\mathbf{y}}^{t+1}-\mathbf{y}_{\mathrm{feas}}, \nabla_\mathbf{y}F(\mathbf{x}^{t+1},\bar{\mathbf{y}}^{t+1})-\bm{\xi}-\mathbf{J}^\top_\vp(\bar{\mathbf{y}}^{t+1})\bar{\vu}^{t+1}\rangle\notag\\
    \leq& F(\mathbf{x}^{t+1},\bar{\mathbf{y}}^{t+1})-F(\mathbf{x}^{t+1},\mathbf{y}_{\mathrm{feas}})-h(\bar{\mathbf{y}}^{t+1})+h(\mathbf{y}_{\mathrm{feas}})-\langle \bar{\vu}^{t+1}, \vp(\bar{\mathbf{y}}^{t+1})-\vp(\mathbf{y}_{\mathrm{feas}})\rangle\notag\\
    \le & \Delta_F + \Delta_h -\langle \bar{\vu}^{t+1}, \vp(\bar{\mathbf{y}}^{t+1})-\vp(\mathbf{y}_{\mathrm{feas}})\rangle, \label{eq:inner-prod1}
\end{align}
where the last inequality follows from \eqref{sectionnotation} and  $\mathbf{y}_{\mathrm{feas}}\in \dom{(h)}$ is as in Assumption~\ref{slatercondition}. 
Also, it follows from the definition of $\bar{\vu}^{t+1}$ that
\begin{equation}\label{eq:inner-prod2}
    \frac{1}{\beta_y}\langle\bar{\vu}^{t+1},\bar{\vu}^{t+1}-\mathbf{u}^{t+1}\rangle=\langle\bar{\vu}^{t+1},\vp(\bar{\mathbf{y}}^{t+1})\rangle.
\end{equation}
Adding \eqref{eq:inner-prod1} and \eqref{eq:inner-prod2} and applying the Cauchy--Schwarz inequality gives
\begin{align}
    &\langle -\vp(\mathbf{y}_{\mathrm{feas}}),\bar{\vu}^{t+1}\rangle\notag\\
    {\leq}&\frac{1}{\beta_y}\|\bar{\vu}^{t+1}\|\|\mathbf{u}^{t+1}\|-\frac{1}{\beta_y}\|\bar{\vu}^{t+1}\|^2+\Delta_F+\Delta_h\notag\\
    \overset{\eqref{in:u}}{\leq}&\frac{-\rho_t}{(1+\rho_t\beta_y)}\|\bar{\vu}^{t+1}\|^2+\frac{\rho_t}{1+\rho_t\beta_y}\|\bar{\vu}^{t+1}\|\|\mathbf{u}^t\|+\frac{1}{1+\rho_t\beta_y}\|\bar{\vu}^{t+1}\|\eta_t+\Delta_F+\Delta_h.\label{eq:tempeq}
\end{align}
Since $p_{\mathrm{feas}}=\min_j(-\vp_j(\mathbf{y}_{\mathrm{feas}}))>0$, $\bar{\vu}^{t+1}\geq \vzero$, we have\begin{equation}\label{eq:lemmaA1eq1}
    \ip{-\vp(\mathbf{y}_{\mathrm{feas}})}{\bar{\vu}^{t+1}}\geq p_{\mathrm{feas}}\|\bar{\vu}^{t+1}\|_1\geq  p_{\mathrm{feas}}\|\bar{\vu}^{t+1}\|.
\end{equation}

Denote $D_u^0=\max(\|\vu^0\|,\frac{\Delta_F+\Delta_h}{p_{\mathrm{feas}}})$ and define $D_u^{t+1}=D_u^t+(\frac{\eta_t}{\rho_t}+\beta_y\eta_t)$ for all $t\ge0$. We prove by induction below that $\|{\mathbf{u}^t}\|\leq D_u^t$ holds for all $t\ge0$. It obviously holds for $t=0$. Suppose that for some $t\geq 0$, this claim holds.  Then, combining~\eqref{eq:tempeq} and~\eqref{eq:lemmaA1eq1}, we obtain
\begin{align*}
   & \left(p_{\mathrm{feas}}+\frac{\rho_t}{1+\rho_t\beta_y}\|\bar{\vu}^{t+1}\|\right)\|\bar{\vu}^{t+1}\|
    \leq \frac{\rho_t}{1+\rho_t\beta_y}\|\bar{\vu}^{t+1}\|(\|\mathbf{u}^t\|+\frac{\eta_t}{\rho_t})+\Delta_F+\Delta_h\\
    \leq&\frac{\rho_t}{1+\rho_t\beta_y}\|\bar{\vu}^{t+1}\|(D_u^t+\frac{\eta_t}{\rho_t})+p_{\mathrm{feas}}D_u^t 
    \leq \left(p_{\mathrm{feas}}+\frac{\rho_t}{1+\rho_t\beta_y}\|\bar{\vu}^{t+1}\|\right)(D_u^t+\frac{\eta_t}{\rho_t}).
\end{align*}
Thus $\|\bar{\vu}^{t+1}\|\leq D_u^t+\frac{\eta_t}{\rho_t}$, which together with \eqref{in:u}, implies and using the triangle inequality, we obtain
\begin{equation}\label{eq:ut1}
    \|\mathbf{u}^{t+1}\|\leq \frac{1}{1+\rho_t\beta_y}(D_u^t+\frac{\eta_t}{\rho_t})+D_u^t\frac{\rho_t\beta_y}{1+\rho_t\beta_y}+\frac{\beta_y\eta_t}{1+\rho_t\beta_y} = D_u^t+(\frac{\eta_t}{\rho_t}+\beta_y\eta_t)\frac{1}{1+\rho_t\beta_y}\leq D_u^{t+1}.
\end{equation}
This completes the induction. Therefore, for any $t\ge0$, it holds
\begin{equation*}
    \|\vu^{t}\|\leq D_u^t=D_u^0+\sum_{i=0}^{t-1}(\frac{\eta_i}{\rho_i}+\beta_y\eta_i) \le D_u^0+\sum_{t=0}^{\infty}(\frac{\eta_t}{\rho_t}+\beta_y\eta_t)= D_u^0+\frac{\eta_0}{\rho_0}\frac{\tau}{\tau-\sigma}+\beta_y\eta_0\frac{\tau}{\tau-1},
\end{equation*}
which finishes the proof of \eqref{eq:boundofueq}. 

Furthermore, to prove $\|\vu^{t,*}\|\leq D_u$, we can apply the above derivation up to~\eqref{eq:ut1} to the exact solution $\vu^{t,*}$, with $\eta_t=0$. 
%This completes the proof. \QED
%
%\subsection{Proof of Lemma~\ref{appendixboundofuhat}}
Finally,    Assumptions~\ref{as1} and \ref{slatercondition} imply the strong duality, i.e., there exists $\hat{\vy}$ such that $
        d(\hat{\vu},\vx)=F(\vx,\hat{\vy})-h(\hat{\vy})=\max_{\vy}F(\vx,\vy)-h(\vy)-\ip{\hat\vu}{\vp(\vy)}.
   $ 
    Hence, by $\hat{\vu}\geq\vzero$ and $\vp(\vy_{\mathrm{feas}})<\vzero$, it follows 
    \begin{equation*}
        F(\vx,\hat{\vy})-h(\hat{\vy})\geq F(\vx,\vy_{\mathrm{feas}})-h(\vy_{\mathrm{feas}})-\ip{\hat{\vu}}{\vp(\vy_{\mathrm{feas}})}\geq F(\vx,\vy_{\mathrm{feas}})-h(\vy_{\mathrm{feas}})+\|\hat{\vu}\|_1{p_{\mathrm{feas}}},
    \end{equation*}
which implies 
    %\begin{equation*}
    $\|\hat{\vu}\|\leq\|\hat{\vu}\|_1\leq\frac{1}{p_{\mathrm{feas}}}\big(F(\vx,\hat{\vy})-F(\vx,\vy_{\mathrm{feas}})-h(\hat{\vy})+h(\vy_{\mathrm{feas}})\big)\leq\frac{\Delta_F+\Delta_h}{p_{\mathrm{feas}}}.$ 
    %\end{equation*}
    This completes the proof. \QED
    
\subsection{Proof of Lemma~\ref{appendixboundofz}}
    Under Assumptions~\ref{as1} and \ref{slatercondition}, strong duality holds for the problem in \eqref{eq:eqproblem}. Hence, there exists an optimal solution $(\hat{\vx}, \vu)$ of \eqref{eq:eqproblem} with a corresponding multiplier $\hat\vz$ such that $ \Psi(\hat{\vx},\vu)+g(\hat{\vx})=\min_{\vx}\Psi(\vx,\vu)+g(\vx)+\ip{\hat{\vz}}{\vc(\vx)}$.
    Hence, by $\hat{\vz}\geq\vzero$ and $\vc(\vx_{\mathrm{feas}})<\vzero$, it follows that
    \begin{equation*}
        \Psi(\hat{\vx},\vu)+g(\hat{\vx})\leq \Psi(\vx_{\mathrm{feas}},\vu)+g(\vx_{\mathrm{feas}})+\ip{\hat{\vz}}{\vc(\vx_{\mathrm{feas}})}\leq \Psi(\vx_{\mathrm{feas}},\vu)+g(\vx_{\mathrm{feas}})-\|\hat{\vz}\|_1{c_{\mathrm{feas}}},
    \end{equation*}
    which implies
    %\begin{equation*}
    $\|\hat{\vz}\|\leq\|\hat{\vz}\|_1\leq\frac{1}{c_{\mathrm{feas}}}\big(\Psi(\vx_{\mathrm{feas}},\vu)-\Psi(\hat{\vx},\vu)+g(\vx_{\mathrm{feas}})-g(\hat{\vx})\big)\leq\frac{\Delta_\Psi+\Delta_g}{c_{\mathrm{feas}}}.$
    %\end{equation*}
 %   This completes the proof. \QED

%\subsection{Proof of Lemma~\ref{appendixboundofz}}
To bound $\vz^t$, for simplicity of notation, we denote
%\begin{equation*}
 $\vv =  \begin{bmatrix}
        \mathbf{x}\\
        \mathbf{u}
    \end{bmatrix}, \text{ and }  \mathbf{v}^{t}\coloneqq\begin{bmatrix}
        \mathbf{x}^{t}\\
        \mathbf{u}^{t}
    \end{bmatrix}.$ 
%\end{equation*}
By \eqref{eq:cond-x-u}, there is $\vxi = \begin{bmatrix}\vxi_\vx \\ \vxi_\vu \end{bmatrix}\in \partial_{(\vx,\vu)} L_{\beta_{x,t}}(\mathbf{x}^{t+1},\mathbf{u}^{t+1},\mathbf{z}^t)$ such that
$
    \dist\left(\vzero , \vxi+\rho_t (\vv^{t+1} - \vv^t)\right)\leq \eta_t.
$
Hence, from the convexity of $L_\beta(\cdot,\cdot,\vz^t)$ and the Cauchy--Schwarz inequality, we have, for any $\vv \in\dom(g) \times \RR^q$,
\begin{align}
    &~  \eta_t  \|
        \mathbf{v}^{t+1}-\mathbf{v}
\|\geq  \langle
        \mathbf{v}^{t+1}-\mathbf{v}
    ,\vxi+\rho_t (\vv^{t+1} - \vv^t) \rangle \notag\\
    \geq &~ L_{\beta_{x,t}}(\mathbf{x}^{t+1},\mathbf{u}^{t+1},\mathbf{z}^t)-L_{\beta_{x,t}}(\mathbf{x},\mathbf{u},\mathbf{z}^t)+\rho_t\langle
        \mathbf{v}^{t+1}-\mathbf{v}, 
        \mathbf{v}^{t+1}-\mathbf{v}^t\rangle \notag \\
     = &~ L_{\beta_{x,t}}(\mathbf{x}^{t+1},\mathbf{u}^{t+1},\mathbf{z}^t)-L_{\beta_{x,t}}(\mathbf{x},\mathbf{u},\mathbf{z}^t) + \frac{\rho_t}{2}\left(\|\mathbf{v}^{t+1}-\mathbf{v}\|^2 + \|\mathbf{v}^{t+1}-\mathbf{v}^t\|^2 - \|\mathbf{v}^{t}-\mathbf{v}\|^2\right). \label{eq:temp001}
\end{align}
Rearranging terms in \eqref{eq:temp001} gives
\begin{align*}
 L_{\beta_{x,t}}(\mathbf{x}^{t+1},\mathbf{u}^{t+1},\mathbf{z}^t)+\frac{\rho_t}{2} 
\|\mathbf{v}^{t+1}-\vv\|^2 
\leq &L_{\beta_{x,t}}(\mathbf{x},\mathbf{u},\mathbf{z}^t)+\frac{\rho_t}{2}\|\mathbf{v}^{t}-\vv\|^2+\eta_t\|\mathbf{v}^{t+1} -\vv\|.
\end{align*}
By the definition of $L_\beta$ in~\eqref{eq:defofL} and using the update $\vz^{t+1}=[\beta_{x,t}c(\vx^{t+1})+\vz^t]_+$ in Algorithm~\ref{al:IAL}, we obtain from the inequality above that
\begin{equation}\label{eq1}
\begin{aligned}
 &~   \Psi(\mathbf{x}^{t+1}, \mathbf{u}^{t+1}) + g(\mathbf{x}^{t+1}) + \frac{1}{2\beta_{x,t}}\|\mathbf{z}^{t+1}\|^2-\frac{1}{2\beta_{x,t}}\|\mathbf{z}^{t}\|^2+\frac{\rho_t}{2} 
\|\mathbf{v}^{t+1}-\vv\|^2 \\
\leq & ~ L_{\beta_{x,t}}(\mathbf{x},\mathbf{u},\mathbf{z}^t)+\frac{\rho_t}{2}\|\mathbf{v}^{t}-\vv\|^2+\eta_t\|\mathbf{v}^{t+1}-\vv\|.
\end{aligned}
\end{equation}

Moreover, it follows from $\vz^{t+1}=[\beta_{x,t}c(\vx^{t+1})+\vz^t]_+$ that
\begin{equation}\label{eq:temp002}
    \langle\mathbf{z}^{t+1},\mathbf{z}^{t+1}-\mathbf{z}^{t}-\beta_{x,t}\mathbf{c}(\mathbf{x}^{t+1})\rangle=0,
\end{equation}
which implies
\begin{equation}\label{eq2}
    \langle\mathbf{z}^{t+1},\mathbf{c}(\mathbf{x}^{t+1})\rangle=\frac{1}{2\beta_{x,t}}\left( 
\|\mathbf{z}^{t+1}\|^2 
- \|\mathbf{z}^{t}\|^2 
+ \|\mathbf{z}^{t+1} - \mathbf{z}^t\|^2 
\right).
\end{equation}
Notice $\mathbf{z}^{t+1}-(\mathbf{z}^{t}+\beta_{x,t}\mathbf{c}(\mathbf{x}^{t+1}))\geq\vzero$. Hence from~\eqref{eq:temp002}, it follows that $\langle \mathbf{z}^{t+1}-\mathbf{z},\mathbf{z}^{t+1}-(\mathbf{z}^{t}+\beta_{x,t}\mathbf{c}(\mathbf{x}^{t+1}))\rangle\leq0$
for any $\vz\geq \vzero$.  
% \begin{equation*}
%     \langle \mathbf{z}^{t+1}-\mathbf{z},\mathbf{z}^{t+1}-(\mathbf{z}^{t}+\beta_{x,t}\mathbf{c}(\mathbf{x}^{t+1}))\rangle\leq0.
% \end{equation*}
%Hence, by identity~\eqref{eq:iden}, we obtain 
Thus
\begin{equation}\label{eq3}
    \frac{1}{2\beta_{x,t}}(\|\mathbf{z}^{t+1}-\mathbf{z}\|^2-\|\mathbf{z}^{t}-\mathbf{z}\|^2+\|\mathbf{z}^{t+1}-\mathbf{z}^t\|^2)\leq\langle\mathbf{z}^{t+1}-\mathbf{z},\mathbf{c}(\mathbf{x}^{t+1})\rangle.
\end{equation}
Adding \eqref{eq1}, \eqref{eq2} and \eqref{eq3} together gives
\begin{align}\label{eq:z-bounded one iteration}
    &\Psi(\mathbf{x}^{t+1}, \mathbf{u}^{t+1}) + g(\mathbf{x}^{t+1}) +\langle \mathbf{z},\mathbf{c}(\mathbf{x}^{t+1})\rangle +\frac{1}{2\beta_{x,t}}\|\mathbf{z}^{t+1}-\mathbf{z}\|^2+\frac{\rho_t}{2} 
\|\mathbf{v}^{t+1}-\vv\|^2 \\
\leq& L_{\beta_{x,t}}(\mathbf{x},\mathbf{u},\mathbf{z}^t)+\frac{1}{2\beta_{x,t}}\|\mathbf{z}^{t}-\mathbf{z}\|^2+\frac{\rho_t}{2}\|\mathbf{v}^{t}-\vv\|^2+\eta_t\|\mathbf{v}^{t+1}-\vv\|\notag.
\end{align}
%holds for any $\mathbf{x}\in\dom(g)$, $\mathbf{u}$ and $\mathbf{z}\geq\vzero$.

Let $\vv^*=(\vx^*,\vu^*)$ be an optimal solution to problem~\eqref{eq:eqproblem} with a corresponding multiplier $\vz^*$. Since $\vz^t\geq 0$ and $(\vx^*,\vu^*)$ is feasible, we have $
    L_{\beta_{x,t}}(\vx^*,\vu^*,\vz^t) \le \Psi(\vx^*,\vu^*)+g(\vx^*).
$ 
Thus letting $\mathbf{x}=\mathbf{x}^*$, $\mathbf{u}=\mathbf{u}^*$ and $\vz=\vz^*$ in \eqref{eq:z-bounded one iteration}, multiplying both sides by $\beta_{x,t}$, and summing it from $t=0$ to $T-1$, we have from
$
    \beta_{x,t}=\beta_{x,0}\sigma^t,
\rho_t=\rho_0\sigma^{-t},
\eta_t=\eta_0\tau^{-t}
$ 
that %from Algorithm~\ref{al:IAL}, we obtain
\begin{align*}
    &\sum_{t=0}^{T-1}\beta_{x,t}\left[\Psi(\mathbf{x}^{t+1}, \mathbf{u}^{t+1}) + g(\mathbf{x}^{t+1}) -\Psi(\mathbf{x}^{*}, \mathbf{u}^{*}) - g(\mathbf{x}^{*})+\langle \mathbf{z}^*,\mathbf{c}(\mathbf{x}^{t+1})\rangle\right] %\\
    %& 
    +\frac{1}{2}\|\mathbf{z}^{T}-\mathbf{z}^*\|^2+\frac{\beta_{x,0}\rho_0}{2} 
\|{\mathbf{v}}^{T} - \vv^*\|^2 \\
\leq &\frac{1}{2}\|\mathbf{z}^{0}-\mathbf{z}^*\|^2+\frac{\beta_{x,0}\rho_0}{2}\|{\mathbf{v}}^{0} -\vv^*\|^2+\beta_{x,0}\eta_0\sum_{t=0}^{T-1}\frac{\sigma^t}{\tau^t}\|{\mathbf{v}}^{t+1} - \vv^*\|.
\end{align*}
By~\cite[Eqn.(2.2)]{xu2021iteration-ialm}, the first term on the left-hand side of the above inequality is nonnegative. Discarding this term and nonnegative term $\|{\mathbf{v}}^{T} - \vv^*\|^2$, we obtain
\begin{equation}\label{eq:bd-zT-z*}
    \|\mathbf{z}^{T}-\vz^*\|^2\leq \|\mathbf{z}^{0}-\vz^*\|^2+\beta_{x,0}\rho_0(D_x^2+4D_u^2)+2\beta_{x,0}\eta_0\frac{\tau}{\tau-\sigma}\sqrt{D_x^2+4D_u^2},
\end{equation}
where we have used $\|\vv^t-\vv^*\|^2 = \|\vx^t - \vx^*\|^2 + \|\vu^t - \vu^*\|^2 \le D_x^2 + 4D_u^2, \forall\, t\ge0$ from \eqref{sectionnotation} and Lemma~\ref{appendixboundofu}, as well as $\sum_{t=0}^{T-1}\frac{\sigma^t}{\tau^t}\le \frac{\tau}{\tau-\sigma}$. 
By Lemma~\ref{appendixboundofz}, we have $\|\vz^*\|\le D'_z$. Thus we obtain the desired result from \eqref{eq:bd-zT-z*} and by noticing $\|\mathbf{z}^{0}-\vz^*\|^2\le \|\vz^0\|^2 + \|\vz^*\|^2$ because $\vz^0\geq\vzero$ and $\vz^*\geq \vzero$. \QED

\section{Smoothness of $\Psi$ and its inexact gradient}\label{proofoflip}
In this section, we give the proof of Lemma~\ref{pr:Lipc} for the Lipschitz continuity of $\nabla\Psi$ and the proof of Lemma~\ref{lemma1}.

\subsection{Proof of Lemma~\ref{pr:Lipc}}
Let $\vx_1$, $\vx_2\in\dom(g)$ and $\vu_1$, $\vu_2$ be given, and let $\vy_i$=$\argmax_\vy Q (\vx_i,\vy,\vu_i)$, $i=1,2$. By the first-order optimality condition, there exists $\chi_1\in\partial h(\vy_1)$ such that
\begin{equation*}
    \nabla_\vy F(\vx_1,\vy_1)-\chi_1-\vJ_\vp({\vy_1})^\top[\beta_y\vp(\vy_1)+\vu_1]_+=\vzero.
\end{equation*}

Hence, it holds that $\ip{\vy_2-\vy_1}{\nabla_y F(\vx_1,\vy_1)-\chi_1-\vJ_\vp({\vy_1})^\top[\beta_y\vp(\vy_1)+\vu_1]_+}=0$. Similarly, there exists $\chi_2\in\partial h(\vy_2)$ such that $\ip{\vy_1-\vy_2}{\nabla_y F(\vx_2,\vy_2)-\chi_2-\vJ_\vp({\vy_2})^\top[\beta_y\vp(\vy_2)+\vu_2]_+}=0$. Summing these two equalities, we obtain
\begin{equation}\label{eq:addineq}
    \ip{\vy_2-\vy_1}{\nabla_y F(\vx_1,\vy_1)-\nabla_y F(\vx_2,\vy_2)-\chi_1+\chi_2-\vJ_\vp({\vy_1})^\top[\beta_y\vp(\vy_1)+\vu_1]_++\vJ_\vp({\vy_2})^\top[\beta_y\vp(\vy_2)+\vu_2]_+}=0.
\end{equation}

The $\mu_F$-strong concavity of $F(\vx_1,\cdot)$ implies that $\ip{\vy_2-\vy_1}{\nabla_y F(\vx_1,\vy_1)-\nabla_y F(\vx_1,\vy_2)}\geq \mu_F\|\vy_2-\vy_1\|^2$. The convexity of the function $h(\cdot)$ yields $\ip{\vy_2-\vy_1}{-\chi_1+\chi_2}\geq 0$. The convexity of $\frac{1}{2\beta_y}\|[\beta_y \vp(\cdot)+\vu_1]_+\|^2$ implies $\ip{\vy_2-\vy_1}{-\vJ_\vp({\vy_1})^\top[\beta_y\vp(\vy_1)+\vu_1]_++\vJ_\vp({\vy_2})^\top[\beta_y\vp(\vy_2)+\vu_1]_+}\geq 0$. Substituting these three inequalities into~\eqref{eq:addineq} yields\begin{equation}\label{eq:lemmac1}
    \ip{\vy_2-\vy_1}{\nabla_y F(\vx_1,\vy_2)-\nabla_y F(\vx_2,\vy_2)-\vJ_\vp({\vy_2})^\top[\beta_y\vp(\vy_2)+\vu_1]_++\vJ_\vp({\vy_2})^\top[\beta_y\vp(\vy_2)+\vu_2]_+}\geq \mu_F\|\vy_2-\vy_1\|^2.
\end{equation}

Note that $\|-\vJ_\vp({\vy_2})^\top[\beta_y\vp(\vy_2)+\vu_1]_++\vJ_\vp({\vy_2})^\top[\beta_y\vp(\vy_2)+\vu_2]_+\|\leq \|\vJ_\vp({\vy_2})\|\|\vu_1-\vu_2\|\leq M_p\|\vu_1-\vu_2\|$. Hence, by the $L_F$ Lipschitz continuity of $F(\cdot,\vy)$ and the Cauchy--Schwarz inequality, we obtain
\begin{align}
    &\ip{\vy_2-\vy_1}{\nabla_y F(\vx_1,\vy_2)-\nabla_y F(\vx_2,\vy_2)-\vJ_\vp({\vy_2})^\top[\beta_y\vp(\vy_2)+\vu_1]_++\vJ_\vp({\vy_2})^\top[\beta_y\vp(\vy_2)+\vu_2]_+}\notag\\
    \leq &\|\vy_2-\vy_1\|(L_F\|\vx_1-\vx_2\|+M_p\|\vu_1-\vu_2\|).\label{eq:lemmac2}
\end{align}
By~\eqref{eq:lemmac1} and~\eqref{eq:lemmac2}, we obtain 
%\begin{equation*}
    $\|\vy_2-\vy_1\|\leq\frac{L_F}{\mu_F}\|\vx_1-\vx_2\|+\frac{M_p}{\mu_F}\|\vu_1-\vu_2\|.$ 
%\end{equation*}
Using the expression of $\nabla\Psi(\vx,\vu)$ given in~\eqref{eq:thenablaofpsi}, for the $\vx$-block, we have
        \begin{align}
        &\|\nabla_\vx \Psi(\vx_1,\vu_1)-\nabla_\vx\Psi(\vx_2,\vu_2)\|^2
        =  \|\nabla_\vx F(\vx_1,\vy_1)-\nabla_\vx F(\vx_2,\vy_2)\|^2 %\notag\\
        \leq %& 
        L_F^2\|\vx_1-\vx_2\|^2+L_F^2\|\vy_1-\vy_2\|^2 \notag\\
        \leq & (L_F^2+\frac{2L_F^4}{\mu_F^2})\|\vx_1-\vx_2\|^2+\frac{2M_p^2L_F^2}{\mu_F^2}\|\vu_1-\vu_2\|^2. \label{eq:lip-psi-x}
    \end{align}
    For the $\vu$-block, we have
        \begin{align}
        &\|\nabla_\vu \Psi(\vx_1,\vu_1)-\nabla_\vu \Psi(\vx_2,\vu_2)\|^2
        = \|\frac{1}{\beta_y}(\vu_1-[\beta_y\vp(\vy_1)+\vu_1]_+)-\frac{1}{\beta_y}(\vu_2-[\beta_y\vp(\vy_2)+\vu_2]_+)\|^2  \notag \\
        \leq & \frac{1}{\beta_y^2}(2\|\vu_1-\vu_2\|+\beta_y\|\vp(\vy_1)-\vp(\vy_2)\|)^2 %\notag\\
        \leq \frac{8}{\beta^2_y}\|\vu_1-\vu_2\|^2+2M^2_p\|\vy_1-\vy_2\|^2\notag\\
        \leq& \frac{4M_p^2L_F^2}{\mu_F^2}\|\vx_1-\vx_2\|^2+(\frac{8}{\beta_y^2}+\frac{4M_p^4}{\mu_F^2})\|\vu_1-\vu_2\|^2.\label{eq:lip-psi-u}
    \end{align}
    Adding \eqref{eq:lip-psi-x} and \eqref{eq:lip-psi-u} shows that $L_\Psi$ given in \eqref{eq:Lip-grad-Psi} is the Lipschitz constant of $\nabla \Psi$. \QED

\subsection{Proof of Lemma~\ref{lemma1}}
By the condition in \eqref{eq:lemma353}, there is $\vxi\in\partial_\vy Q(\vx,\vy,\vu)$ such that $\|\vxi\|\leq\frac{\delta\mu_F}{(L_F+M_p)}$.
Let $\vy^*=\vy(\vx,\vu)$ by the definition in~\eqref{eq:solve y}. Then %$\vzero\in\partial_\vy Q(\vx,\vy^*,\vu)$.  
 by the $\mu_F$-strong concavity of $Q(\vx,\cdot,\vu)$, we obtain
%\begin{align*}
 $\mu_F\|\vy-\vy^*\|^2 \le   \ip{\vy-\vy^*}{\vzero-\vxi}\le \|\vxi\|\|\vy-\vy^*\|$,
%\end{align*}
which implies
%By the Cauchy--Schwarz inequality, we have
\begin{equation}\label{eq:lemma354}
    \|\vy-\vy^*\|\leq \frac{1}{\mu_F}\|\vxi\|\leq \frac{\delta}{(L_F+M_p)}.
\end{equation}
Hence from~\eqref{eq:thenablaofpsi} and the Lipschitz continuity of $\nabla F$ and $\mathbf{p}$, it follows
\begin{align*}
    &\|\nabla\Psi(\vx,\vu)-\nabla_{(\vx,\vu)} Q(\vx,\vy,\vu)\|\\
    =&\left\|\left[\nabla_{\mathbf{x}} F(\mathbf{x},\vy^*),\tfrac{1}{\beta_y}(\mathbf{u} -[\beta_y\,\mathbf{p}(\vy^*) + \mathbf{u}]_+)\right]-\left[\nabla_{\mathbf{x}} F(\mathbf{x},\vy),\tfrac{1}{\beta_y}(\mathbf{u} -[\beta_y\,\mathbf{p}(\vy) + \mathbf{u}]_+)\right]\right\|\\
    \leq & \|\nabla_{\mathbf{x}} F(\mathbf{x},\vy^*) - \nabla_{\mathbf{x}} F(\mathbf{x},\vy)\|+\|\tfrac{1}{\beta_y}[\beta_y\,\mathbf{p}(\vy) + \mathbf{u}]_+
            - \tfrac{1}{\beta_y}[\beta_y\,\mathbf{p}(\vy^*) + \mathbf{u}]_+\|\\
    \leq &\|\nabla_{\mathbf{x}} F(\mathbf{x},\vy^*) - \nabla_{\mathbf{x}} F(\mathbf{x},\vy)\|+\|\vp(\vy)-\vp(\vy^*)\|\\
    \leq & (L_F + M_p)\|\vy^* - \vy\| %\\
    \leq  \delta.
\end{align*}
This completes the proof. \QED

% \begin{lemma}
% \phantomsection
% \label{pr:Lipcnu}
%     Let $\Psi_\nu(\vx,\vu)$ be defined in~\eqref{eq:primalnu}. Under Assumption~\ref{as1}--\ref{slatercondition}. $\nabla\Psi_\nu(\vx,\vu)$ is Lipschitz continuous with constant $L_{\Psi,\nu}$, where
% \begin{equation}\label{eq:lip-nu-psi}
%     L_{\Psi,\nu}=\max\{\sqrt{L_F^2+\frac{2L_F^4}{\nu^2}+\frac{4M_p^2L_F^2}{\nu^2}},\sqrt{\frac{2M_p^2L_F^2}{\nu^2}+\frac{8}{\beta_y^2}+\frac{4M_p^4}{\nu^2}}\}.
% \end{equation}
% \end{lemma}
% \begin{proof}
%     This is a direct consequence of Lemma~\ref{pr:Lipc} after replacing $F$ with $F_\nu$. Note that Lemma~\ref{pr:Lipc} only uses the Lipschitz constant of $F(\cdot,\vy)$. Since $F_\nu(\cdot,\vy)$ has the same Lipschitz property as $F(\cdot,\vy)$, the constant $L_F$ does not need to be modified.
% \end{proof}

\bibliographystyle{siam}
\bibliography{optim, Pub_YXu}
			
\end{document}